\documentclass{article}
\usepackage[T1]{fontenc}
\usepackage[english]{babel}
\usepackage{amsmath, amssymb}
\usepackage{amsfonts}
\usepackage{amsthm}
\usepackage{amscd}
\usepackage{mathrsfs} 
\usepackage[shortlabels]{enumitem}
\usepackage{mathrsfs, euscript}
\usepackage{lmodern}
\DeclareMathOperator*{\argmax}{arg\,max}
\DeclareMathOperator*{\argmin}{arg\,min}
\usepackage{algorithmic}

\usepackage{algorithm}
\usepackage{graphicx}
\usepackage{tablefootnote}
\usepackage{mwe}
\usepackage{subfig}
\usepackage{caption}
\usepackage{array}      
\usepackage{ragged2e}   
\usepackage{siunitx}    
\usepackage[font=small,labelfont=bf]{caption}
\usepackage{bm}
\usepackage[capitalize]{cleveref}

\newtheorem{lem}{Lemma}[section]
\newtheorem{cor}{Corollary}[section]
\newtheorem{rem}{Remark}[section]
\newtheorem{prop}{Proposition}[section]

\crefname{ALC@unique}{Step}{Steps} 
\crefname{equation}{}{}

\numberwithin{equation}{section}

\newcommand{\ad}{A_{\mathcal{G}}}
\newcommand{\G}{\mathcal{G}}
\newcommand{\bG}{\bar{\mathcal{G}}}
\newcommand{\bE}{\bar{E}}

\newcommand{\n}[1]{\|#1\|}
\newcommand{\p}{\mathcal{P}_s}
\newcommand{\R}{\mathbb{R}}
\newcommand{\N}{\mathbb{N}}
\newcommand{\supp}{\tx{supp}}
\newcommand{\M}{\mathcal{M}}
\newcommand{\Sc}[2]{ #1^{\intercal} #2}

\newcommand{\cC}{c}
\newcommand{\hg}{h_{\G}}
\newcommand{\bc}{\bar{C}}
\newcommand{\bs}{\bar{S}}
\newcommand{\tx}{\text}

\newcommand{\sm}{\setminus}
\newcommand{\F}{\mathcal{F}}
\newcommand{\Q}{\mathcal{Q}}
\newcommand{\FW}{\mathcal{FW}}
\newcommand{\FD}{\mathcal{FD}}
\newcommand{\fun}{f}

\title{Fast cluster detection in networks by first-order optimization}
\author{ Immanuel~M.~Bomze\thanks{ISOR, VCOR \& ds:UniVie,
		Universit\"{a}t Wien, Austria
		(\tt{immanuel.bomze@univie.ac.at})}
	\and
	Francesco~Rinaldi\thanks{Dipartimento di Matematica ``Tullio Levi-Civita'', Universit\`a
		di Padova, Italy
		(\tt{rinaldi@math.unipd.it})}
	\and 
	Damiano~Zeffiro\thanks{Dipartimento di Matematica ``Tullio Levi-Civita'', Universit\`a
		di Padova, Italy (\tt{damiano.zeffiro@math.unipd.it})}
}

\begin{document}
	
	\maketitle
	
	\begin{abstract} 
		Cluster detection plays a fundamental role in the analysis of data. In this paper, we focus on the use of $s$-defective clique models for network-based cluster detection and propose a nonlinear optimization approach that efficiently handles those models in practice. In particular, we introduce an equivalent continuous formulation for the problem  under analysis, and we analyze some tailored variants of the Frank-Wolfe algorithm that enable us to quickly find maximal $s$-defective cliques. The good practical behavior of those algorithmic tools, which is closely connected to their support identification properties, makes them very appealing in practical applications.  The reported numerical results clearly  show the effectiveness of the proposed approach. 
	
		\end{abstract}

\section{Introduction}
In the context of network analysis the clique model, dating back at least to the work of Luce and Perry \cite{luce1949method} about social networks, refers to subsets with every two elements in a direct relationship. The problem of finding maximal cliques has numerous applications in domains including telecommunication networks, biochemistry, financial networks, and scheduling (\cite{bomze1999maximum}, \cite{wu2015review}). From an optimization perspective, this problem has been the subject of extensive studies stimulating new research directions in both continuous and discrete optimization (see, e.g., \cite{bomze1998standard}, \cite{bomze1999maximum}, \cite{bomze2000copositive}, \cite{stozhkovcontinuous}). The Motzkin-Straus quadratic formulation \cite{motzkin1965maxima} in particular has motivated several algorithmic approaches (see \cite{bomze1997evolution}, \cite{hungerford2019general} and references therein) to the maximum clique problem, beside being of independent interest for its connection with Tur\'{a}n's theorem \cite{aigner2010proofs}. \\
Since the strict requirement that every two elements have a direct connection is often not satisfied in practice, many relaxations of the clique model have been proposed (see, e.g., \cite{pattillo2013clique} for a survey). In this work we are interested in $s$-defective cliques (\cite{chen2021computing}, \cite{trukhanov2013algorithms}, \cite{yu2006predicting}), allowing up to $s$ links to be missing, and introduced in \cite{yu2006predicting} for the analysis of protein interaction networks obtained with large scale techniques subject to experimental errors. \\ 
In this paper, we first define a regularized version of a cubic continuous formulation for the maximum \mbox{$s$-defective} clique problem  proposed in \cite{stozhkovcontinuous}. We then apply variants of the classic Frank-Wolfe (FW) method  \cite{frank1956algorithm} to this formulation. \\ 
FW variants are a class of first order optimization methods widely used in the optimization and machine learning communities (see, e.g., \cite{clarkson2010coresets}, \cite{jaggi2013revisiting}, \cite{lacoste2015global}, \cite{kerdreux2020affine} and references therein) thanks to their sparse approximation properties, their weaker requirement of a linear minimization oracle instead of a projection oracle with respect to proximal gradient like methods, and their ability to quickly identify the support of a solution. These identification properties, first proved qualitatively for the Frank Wolfe method with in face directions (FDFW, \cite{bashiri2017decomposition}, \cite{freund2017extended}, \cite{guelat1986some}) in the strongly convex setting \cite{guelat1986some}, were recently revisited for the away-step Frank-Wolfe (AFW)  with quantitative bounds (\cite{bomze2020active}, \cite{garber2020revisiting}) and extended to non convex objectives (\cite{bomze2019first}, \cite{bomze2020active}). \\
As we will see in the paper, the support identification properties of FW variants are especially suited for our maximal $s$-defective clique formulation, since in this case the optimization process can stop as soon as the support of a solution is identified. \\
Our contributions can be summarized as follows:
\begin{itemize}
	\item We solve the spurious solution problem for the maximum $s$-defective clique formulation proposed in \cite{stozhkovcontinuous} by introducing a regularized version, for which we prove equivalence between local maximizers and maximal $s$-defective cliques. In particular, no postprocessing algorithms are needed to derive the desired structure from a local solution. Our work develops along the lines of analogous results proved for regularized versions of the Motzkin - Straus quadratic formulation (\cite{bomze1999maximum}, \cite{hungerford2019general}).
	\item We prove that the FDFW applied to our formulation identifies the support of a maximal $s$-defective clique in a finite number of iterations.
	\item We propose a tailored Frank-Wolfe variant for the $s$-defective clique  formulation at hand exploiting its product domain structure. This method retains the identification properties of the FDFW while significantly outperforming it in numerical tests.  
\end{itemize}
The paper is organized as follows: after 
giving some basic notation and preliminaries
in Section \ref{Sec:Not}, we study the regularized maximum $s$-defective clique formulation in Section \ref{Sec:Reg}. We then analyze, in Section \ref{Sec:FDFW}, the FDFW algorithm and prove that it identifies a maximal $s$-defective clique in finite time. In Section \ref{Sec:FWvar}, we describe our tailored FW variant and prove that it shares similar identification properties as FDFW. Finally, in Section~\ref{Sec:NumRes}, we report some numerical results showing the practical effectiveness of the proposed approach.  In order to improve readability,
some technical details and numerical result tables are deferred to a small appendix.
\section{Notation and preliminaries}\label{Sec:Not}
For any integers $a,b$ we denote by $[a\! :  \! b]$ the set of all integers $k$ satisfying $a\le k \le b$. 

For a vector $r \in \R^d$, the $d$-dimensional Euclidean space, and a set $A \subset [1\!:\! d]$,  we denote with $r_A$ the components of $r$ with indexes in $A$; $e$ is always a vector with all components equal to 1, and dimension clear from the context. Similarly, we denote by $e_i$ the $i$-th column of an appropriately sized identity matrix. 

Let $\G = (V, E)$ be a graph with  vertices $V$ and  and edges $E$, $n = |V|$, 
$\ad$ the adjacency matrix of $\G$, and let  $\bG = (V, \bE)$ the complementary graph.  
The notation we use largely overlaps with the one introduced in \cite{stozhkovcontinuous}. For $s \in \mathbb{N}$ with $s\le |\bE|$ we define 
$$D_s(\G) = \{  y \in  \{0, 1\}^{\bE} \ | \  e^{\intercal}y 
\leq s \}\, ,$$
representing the set of "fake edges" to be added to the graph in order to complete an $s$-defective clique, and its continuous relaxation as
$$D'_s(\G) =  \{  y \in  [0, 1]^{\bE} \ | \  e^{\intercal}y 
\leq s \} \, .$$ For $y \in D'_s(\G)$ we define the induced adjacency matrix $A(y)\in \R^{n \times n}$ as	
$$
A(y)_{ij} = 
\begin{cases}
	y_{ij} \ &\text{if} \ \{i, j\} \in \bE \, ,	\\
	0 \ &\text{if} \	\{i, j\} \notin E \, .
\end{cases}
$$
For $y \in D_s(\G)$ in particular we define $\G(y)$ as the graph with adjacency matrix $A_\G + A(y)$, that is the  graph where we add to $\G$ the edge $\{i, j\}$ whenever $y_{ij} = 1$. We also define $E(i)$ and $E^y(i)$ as the neighbors of $i$ in $\G$ and $\G(y)$ respectively. \\  
Let $\p=\Delta_{n-1} \times D'_s(\G) $, with $\Delta_{n-1}$ the 
$(n-1)$-dimensional simplex.
The objective of the $s$-defective clique relaxation defined in~\cite{stozhkovcontinuous} is 
\begin{equation}
	f_{\G}(z)=	f_{\G}(x, y) := x^{\intercal}[\ad + A(y)]x \,  ,\quad z=(x,y)\in \p
\end{equation}
so that when $A(y) = 0$ one retrieves Motzkin-Straus quadratic formulation. \\
\section{A regularized maximum $s$-defective clique formulation}\label{Sec:Reg}
Here we consider the problem     
\begin{equation} \label{deq:P}
	\max\{ h_{\G}(z)\ | \ 	z\in \p \}\, , \tag{P} 	
\end{equation}
where $h_{\G}: \p \rightarrow \R_{> 0}$ is a regularized version of $f_{\G}$:
$$h_{\G}(z)= h_{\G}(x, y) := f_{\G}(x, y) + \frac{\alpha}{2} \n{x}^2 + \frac{\beta}{2} \n{y}^2 $$
for some $\alpha \in (0, 2)$ and $\beta > 0$. In particular, when $y = 0$ the objective $h_{\G}$ corresponds to  the quadratic regularized maximal clique formulation introduced in \cite{bomze1997evolution}. \\
For non-empty $C \subseteq V$ let $x^{(C)} = \frac 1{|C|}\, \sum_{i\in C}e_i$
be the characteristic vector in $\Delta_{n  - 1}$ of the clique $C$, 	and $$\Delta^{(C)}= \{ x\in \Delta_{n-1} \ | \ x_i=0\mbox{ for all }i\in V\setminus C\}$$ 
be the minimal face of $\Delta_{n  - 1}$ containing  $x^{(C)}$ in its relative interior. \\
For $p \in \p$ we define as $T_{\p}(p)= \{ v-p:v\in \p\}$ as the cone of feasible directions at $p$ in $\p$, while for $r \in \mathbb{R}^{|V| + |\bE|}$ we define  $T_{\p}^0(p, r)$ as the intersection between $T_{\p}(p)$ and the plane orthogonal to $r$:
\begin{equation*}
	T_{\p}^0(p, r) = \{d \in T_{\p}(p) \ | \ \Sc{d}{r} = 0\} \, .
\end{equation*}
We now prove that there is a one to one correspondence between (strict) local maxima of $h_{\G}$ and $s$-defective cliques coupled together with $s$ fake edges including the one missing on the clique.\\ Recall that in our 
polytope-constrained setting, (second order) sufficient conditions for the local maximality of $p \in \p$ are (see, e.g., \cite{bertsekas1997nonlinear}) 
\begin{equation}
	\Sc{\nabla \hg(p)}{d} \leq 0 \tx{ for all } d \in T_{\p}(p)
\end{equation}
and 
\begin{equation}
	d^{\intercal} D^2\hg(p) d  < 0 \tx{ for all } d \in T_{\p}^0(p, \nabla \hg(p)) \, .
\end{equation}
In the rest of the article we use $\M_s(\G)$ to denote the set of strict local maximizers of $h_{\G}$.
\begin{prop}[characterization of local maxima for $h_{\G}$] \label{p:lm}
	The following are equivalent: 
	\begin{itemize}
		\item[(i)] $p \in \p$ is a local maximizer for $h_{\G}(x,y)$;
		\item[(ii)] $p \in \M_s(\G)$;
		\item[(iii)] $p = (x^{(C)}, y^{(p)})$ where   
		$s=e^{\intercal}y^{(p)} \in \N$, with $C$ an $s$-defective clique in $\G$ which is also a maximal clique in $\G(y^{(p)})$, and $y^{(p)} \in D_s(\G)$ such that  $y^{(p)}_{ij} = 1$ for every $\{i, j\} \in {C \choose 2} \cap\bE$.		
	\end{itemize}
	In either of these equivalent cases, we have
	\begin{equation}\label{loval}h_{\G}(p) = 1 - \frac{2 - \alpha}{2|C|} + s\,  \frac{\beta}{2} \, .
	\end{equation}
\end{prop}
\begin{proof}
	Let $p = (x^{(p)}, y^{(p)}) \in \p$, $g = \nabla \hg(p), H = D^2\hg(p)$. \\
	(ii) $\Rightarrow$ (i) is trivial. \\	
	(i) $\Rightarrow$ (iii). If $s:= e^{\intercal}
	{y^{(p)}}$ were fractional, then for some $\{i, j\} \in \bar{E}$ we would have $y^{(p)}_{ij} < 1$. Furthermore 
	\begin{equation} \label{eq:dph}
		\frac{\partial\hg (p)}{\partial y_{ij}}  = 2x^{(p)}_i x^{(p)}_j + \beta y^{(p)}_{ij} \geq 0, \quad 	\frac{\partial\hg (p)}{\partial^2y_{ij}}   = \beta > 0 \, .
	\end{equation}
	Thus for $\varepsilon >0$ small enough we have $\hg (p + \varepsilon e_{ij}) > \hg(p)$ with $p + \varepsilon e_{ij} \in \p$, which means that $p$ is not a local maximizer. Hence $s\in \N$ and obviously $s\le |\bE|$ as well as $y^{(p)}\in D'_s(\G) $.\\
	Assume now  by contradiction that $p$ is a local maximizer but $y^{(p)} \notin D_s(\G)$. Then for two distinct edges $\{i, j\}$, $\{l, m\} \in \bar{E}$ we must have $y^{(p)}_{ij}, y^{(p)}_{lm} \in (0, 1)$. Let $d = (0, e_{ij} - e_{lm})$. Since $\pm d$ are both feasible directions and $p$ is a local maximizer, necessarily $\Sc{g}{d} = 0$. But we also have
	\begin{equation}
		d^{\intercal} H d  = \frac{\partial \hg(p)}{\partial^2 y_{ij}} + \frac{\partial \hg(p)}{\partial^2 y_{lm}} - 2\, \frac{\partial \hg(p)}{\partial y_{ij} \partial y_{lm}}  = 2\beta > 0 \, .
	\end{equation}  
	so that again for $\varepsilon>0$ small enough $\hg (p + \varepsilon d) > \hg(p)$ with $p + \varepsilon d \in \p$, a contradiction. \\
	We proved that if $p$ is a local maximizer, then $s=e^{\intercal} {y^{(p)}} \in \N$ and  $y^{(p)} \in D_s(\G)$. But $x^{(p)}$ must be a local maximizer for the function $x \mapsto h_{\G}(x, y^{(p)})$, which is (up to a constant) a regularized maximal clique relaxation for the augmented graph $\G (y^{(p)})$. Thus by well known results (see, e.g., \cite{hungerford2019general}, \cite{bomze1997evolution}) we must have $x = x^{(C)}$ with $C$ a maximal clique in $\G(y^{(p)})$. In particular, since $\G(y^{(p)})$ is defined by adding $s$ edges to $\G$, $C$ must be an $s$-defective clique in $\G$. \\
	(iii) $\Rightarrow$ (ii). For a fixed $p = (x^{(C)}, y^{(p)})$ with $C, y^{(p)}$ satisfying the conditions of point (iii) let $\bar{C} = V\sm C$, $S = \tx{supp}(y^{(p)})$ and $\bar{S} = \bar{E} \sm S$. We abbreviate $E^{(p)}(i)=E^y(i)$ with $y=y^{(p)}$. For every $i \in V$ we have 
	\begin{equation}
		g_i = \alpha x^{(C)}_i + \sum_{j \in E^{(p)}(i)} 2x^{(C)}_j
	\end{equation}
	In particular for $i \in C$ 
	\begin{equation}\label{peq:gic}
		g_i = \frac{\alpha}{|C|} + \sum_{j \in C \sm \{ i \}} 2x^{(C)}_j = \frac{1}{|C|}(\alpha + 2|C| - 2)
	\end{equation}
	and for every $i \in \bc$
	\begin{equation} \label{peq:gbic}
		g_i =  \sum_{j \in E^{(p)}(i)\cap C} 2x^{(C)}_j \le \frac{2|C|-2}{|C|}
	\end{equation}
	where we used $x^{(C)}_j = 1/|C|$ for every $j \in C$, $x^{(C)}_j = 0$ otherwise. \\
	For $\{i, j\} \in \bE$ we have 
	\begin{equation}
		g_{ij} = \beta y^{(p)}_{ij} + 2x^{(C)}_i x^{(C)}_j
	\end{equation}
	and in particular $g_{ij} = 0$ for $\{i, j\} \in \bar{S}$, while for $\{i, j\} \in S$
	\begin{equation}
		g_{ij} = \beta + 2 x^{(C)}_i x^{(C)}_j \geq \beta > 0 \, ,
	\end{equation}
	where we used $y^{(p)}_{ij} = 1$ for $\{i, j\} \in S$, $0$ otherwise, and 
	$x^{(C)}_i x^{(C)}_j = 0$ for $\{i, j\} \in \bar{S}\subseteq \bE$. \\
	Let $d$ be a feasible direction from $p$, so that $d = v - p$ with $v \in \p$. Let $\sigma_S = \sum_{\{i, j\} \in S} g_{ij}$, $\sigma_C = \sum_{i \in C} v_i $ \\ $= 1 - \sum_{i \in \bc} v_i \in [0, 1]$, $m_{\bar{C}} = \max_{i \in \bar{C}}g_i$, so that by \cref{peq:gbic} we have $m_{\bar{C}} \le \frac{2|C|-2}{|C|}$.
	Then
	\begin{equation} \label{peq:gp}
		\Sc{g}{p} = \sum_{i \in \bar{C}} x^{(C)}_i g_i + \sum_{i \in C} x^{(C)}_i g_i + \sum_{(i, j ) \in S} y^{(p)}_{ij} g_{ij} = \frac{1}{|C|} \sum_{i \in C} g_i +  \sum_{\{i, j\} \in S} g_{ij} = \frac{1}{|C|}(\alpha + 2|C| - 2) + \sigma_S
	\end{equation}
	where we used \cref{peq:gic} in the last equality. We also have
	\begin{equation} \label{peq:gv}
		\Sc{g_V}{v_V} =	\Sc{g_C}{v_C} + \Sc{g_{\bc}}{v_{\bc}} \leq \frac{\alpha + 2|C| - 2}{|C|} \sigma_C + (1 - \sigma_C) m_{\bc} \leq \frac{\alpha + 2|C| - 2}{|C|} 
	\end{equation}
	where we used \cref{peq:gic} together with the H\"older inequality in the first inequality, $m_{\bc} \le \frac{2|C|-2}{|C|}$ in the second inequality and $\sigma_C \leq 1$. Finally, 
	\begin{equation} \label{peq:ge}
		\Sc{g_{\bar{E}}}{v_{\bar{E}}} = \Sc{g_S}{v_S} + \Sc{g_{\bs}}{v_{\bs}} = \Sc{g_S}{v_S} \leq \sigma_S
	\end{equation}
	where we used $g_{\bar{S}} = 0$ in the second equality, and $v_i \leq 1$ for every $i \in \bE$ in the inequality. We can conclude 
	\begin{equation} \label{peq:1c}
		\Sc{g}{d} = \Sc{g_V}{v_V} + \Sc{g_{\bar{E}}}{v_{\bar{E}}} - \Sc{g}{p} \leq 0  
	\end{equation}
	where we used \cref{peq:ge}, \cref{peq:gp} and \cref{peq:gv} in the inequality. We have equality iff there is equality both in \cref{peq:gv} and \cref{peq:ge}, and thus iff $v = (x^{(v)}, y^{(v)})$ with $\tx{supp}(x^{(v)}) = C$ and $y^{(v)} = y^{(p)}$. In particular $p$ is a first order stationary point with 
	\begin{equation}
		T_{\p}^0(p, g) = \{d \in T_{\p}(p) \ | \ d = v - p, v_{\bc} = 0, v_{\bar{E}} = p_{\bar{E}} \} = \{d \in T_{\p}(p) \ | \ d_{\bc} = d_{\bar{E}} = 0\} \, .
	\end{equation}
	Let $H_C$ be the submatrix of $H$ with indices in $C$. We have, for $(i, j) \in C^2$ with $i \neq j$, $H_{ij} = 1$ since $C$ is a clique in the augmented graph $\G(y_p)$, while $H_{ii} = \alpha$ for every $i \in V$ and in particular for every $i \in C$. This proves 
	\begin{equation}
		H_C = 2e e^{\intercal} + (\alpha - 2) \mathbb{I} \, .
	\end{equation}
	Now if $T_{\p}^0(p, g) \ni d \neq 0$ we have 
	\begin{equation}\label{peq:2c}
		d^{\intercal} H d = d_C^{\intercal} H_C d_C = d_C^{\intercal}(2e e^{\intercal} + (\alpha - 2) \mathbb{I}) d_C = (\alpha - 2) \n{d_C}^2 < 0
	\end{equation}
	where we used $d_{\bc} = d_{\bE} = 0$ in the first equality, $e^{\intercal} d_C = e^{\intercal} (v_V-p_V)=1-1=0$ in the third one. This proves the claim, since we have sufficient conditions for local optimality thanks to \cref{peq:1c} and \cref{peq:2c}. 
\end{proof}
As a corollary, the global optimum of $h_{\G}$ is achieved on maximum $s$-defective cliques.
\begin{cor} \label{c:globalm}
	The global maximizers of $h_{\G}(z)$ are all the points $p$ of the form $p = (x^{C^*}, y^{(p)})$ where $C^*$ is an $s$-defective clique of maximum cardinality, and $y^{(p)} \in D_s(\G)$ such that $e^{\intercal} y^{(p)} = s$.
\end{cor}
\begin{proof}
	Let $p = (x^{(C)}, y^{(p)})$ a local maximizer for $h_{\G}(z)$. Then its objective value is, by~\cref{loval},
	$h_{\G}(p) = 1 - \frac{2 - \alpha}{2|C|} + s\, \frac{\beta}{2} $, which is
	(globally) maximized when $|C|$ is as large as possible, because $2 - \alpha > 0$ by assumption.
\end{proof}
Thanks to \cref{p:lm}, for every $p \in \M_s(\G)$ we can define $y^{(p)} \in D_s(\G) $ and a maximal clique $C$ of $\G(y^{(p)})$ such that $p = (x^{(C)}, y^{(p)})$. We now recall that the face of a polytope $\Q$ exposed by a gradient $g \in \R^n$ is defined as
\begin{equation}
	\F_e(g) = \argmax_{w \in Q} \Sc{g}{w}.
\end{equation}
With this notation, we prove that the face of $\p$ exposed by the gradient in $p \in \M_s(\G)$ is simply the product between $\Delta^{(C)}$ and the singleton $\{y^{(p)}\}$. This property, sometimes referred to as strict complementarity, is of key importance to prove identification results for Frank-Wolfe variants (see \cite{bomze2019first}, \cite{bomze2020active}, \cite{garber2020revisiting}, and the discussion of external regularity in~\cite[Section~5.3]{bomze2002sirev}).
\begin{lem}\label{l:sc}
	Let $p = (x^{(C)}, y^{(p)}) \in \M_s(\G)$. Then the face exposed by $\nabla h_{\G}(p)$ coincides with the minimal face $\F(p)$ of $\p$ containing $p$:
	\begin{equation} \label{eq:fep}
		\F_{e}(\nabla h_{\G}(p)) = \F(p) = \Delta_{n - 1}^{(C)} \times \{y^{(p)}\} \, .
	\end{equation}
\end{lem}
\begin{proof}
	To start with, the second equality follows from the fact that $y^{(p)}$ is a vertex of $D_s'(\G)$ and that $\Delta_{n - 1}^{(C)}$ is the minimal face of $\Delta_{n  - 1}$ containing $x^{(C)}$. The first equality is then equivalent to proving that for every vertex $a = (a_x, a_y)$ of $\p$ with $a\in \p \sm \F(p)$ we have $\lambda_a(p) < 0$. Given that stationarity conditions must hold in $\Delta_{n - 1}$ and $D_s'(\G)$ separately,
	$\lambda_a(p) < 0$ if and only if
	\begin{subequations}
		\begin{align}
			\lambda^x_{a}(p):=& \Sc{\nabla_x h_{\G}(p)}{(a_x - x^{(C)})} \leq 0\, , \label{eq:xc}\\
			\lambda^{y}_a(p):= &\Sc{\nabla_y h_{\G}(p)}{(a_y - y^{(p)})} \leq 0 \, , \label{eq:yc}
		\end{align}	
	\end{subequations}
	and at least one of these relations must be strict. Since $a$ is a vertex of $\p$,  $a_x = e_l$ with $l \in [1:n]$ and $a_y \in D_s(\G)$, while  $a \notin \F(p)$ implies $l \notin C$ or $a_y \neq y^{(p)}$. If $l \in C$ then $\lambda^x_{a}(p) = 0$ by stationarity conditions, otherwise
	\begin{equation} \label{eq:x1piece}
		\Sc{\nabla_x h_{\G}(p)}{x^{(C)}} = 2 (x^{(C)})^{\intercal}[A + A(y^{(p)})] x^{(C)} + \alpha \n{x^{(C)}}^2 = 2 - \frac{2 - \alpha}{|C|}
	\end{equation}  
	and
	\begin{equation} \label{eq:x2piece}
		\Sc{\nabla_x h_{\G}(p)}{a_x} = \frac{\partial}{\partial x_l} h_{\G}(p) = \alpha x_l  + \sum_{j \in C \cap E^{(p)}(l)}2 x_j = 2\frac{|C \cap E^{(p)}(l)|}{|C|} \leq 2 - \frac{2}{|C|} \, ,
	\end{equation} 
	where we used $a_x = e_l$ in the first equality, $l \notin C$ together with $x_j = 1/|C|$ for every $j \in C$ in the third equality, and the maximality of the clique $C$ in the augmented graph $\G(y^{(p)})$ in the inequality. Combining \cref{eq:x1piece} and \cref{eq:x2piece}, we obtain
	\begin{equation}
		\Sc{\nabla_x h_{\G}(p)}{(a_x - x^{(C)})} \leq - \frac{\alpha}{|C|} < 0 \, ,
	\end{equation}
	which proves that \cref{eq:xc} holds with strict inequality if $l\notin C$, or else with equality if $l\in C$. \\
	In a similar vain we proceed with~\cref{eq:yc}. If $a_y=y^{(p)}$ then~\cref{eq:yc} holds with equality but then $l\in V\sm C$ and we are done. So suppose $a_y\neq y^{(p)}$, and consider the supports $S_y = \{\{i, j\} \in \bar{E} \ | \ (a_y)_{ij} = 1\}$ and $S_p = \{\{i, j\} \in \bar{E} \ | \ y^{(p)}_{ij} = 1\}$. Since $a_y\in D_s(\G)$, we have $|S_y|\leq s$ and on the other hand, by Proposition~\ref{p:lm}(iii), $|S_p| =s$. As $S_y$ and $S_p$ must be distinct, we conclude $S_y\sm S_p\neq \emptyset$.
	Furthermore, by~\cref{eq:dph} for every $\{i, j\}$ in $A_p$ we have
	\begin{equation} \label{eq:apy1}
		\frac{\partial}{\partial y_{ij}} h_{\G}(p) \geq \beta y^{(p)}_{ij} = \beta > 0 \, ,
	\end{equation}
	while for every $\{i, j\}$ in $A_y \sm A_p$ we have
	\begin{equation} \label{eq:apy2}
		\frac{\partial}{\partial y_{ij}} h_{\G}(p)  = 0
	\end{equation}
	because $y^{(p)}_{ij} = 0$ by definition of $A_p$ and $x^{(C)}_{i} x^{(C)}_j = 0$ since, again invoking~\cref{p:lm}(iii), $\{i, j\}\in \bar E\setminus {C \choose 2}$. So we can finally prove \cref{eq:yc} by observing
	\begin{equation}
		\begin{aligned}
			& \Sc{\nabla_y h_{\G}(p)}{(a_y - y^{(p)})}  = \sum_{\{i, j\} \in A_y} \frac{\partial}{\partial y_{ij}} h_{\G}(p) - \sum_{\{i, j\} \in A_p} \frac{\partial}{\partial y_{ij}} h_{\G}(p) \\ = & \sum_{\{i, j\} \in A_y \sm A_p} \frac{\partial}{\partial y_{ij}} h_{\G}(p) - \sum_{\{i, j\} \in A_p \sm A_y} \frac{\partial}{\partial y_{ij}} h_{\G}(p) = 
			-  \sum_{\{i, j\} \in A_p \sm A_y} \frac{\partial}{\partial y_{ij}} h_{\G}(p) < 0
		\end{aligned}
	\end{equation}
	where we used \cref{eq:apy2} in the third equality and \cref{eq:apy1} together with $A_p\sm A_y \neq \emptyset$ in the inequality. 
\end{proof}

\section{Frank-Wolfe method with in face directions} \label{Sec:FDFW}

Let $\Q = \tx{conv}(A) \subset \R^n $ with $|A| < +\infty$.  In this section, we consider the FDFW for the solution of the smooth constrained optimization problem
\begin{equation*}
	\max \{ f(w) \ | \ w \in \Q \} \, .
\end{equation*}
In particular, $\{w_k\}$ is always a sequence generated by the FDFW applied to the polytope $\Q$ with objective $f$.
For $w \in \Q$ we denote with $\F(w)$ the minimal face of $\Q$ containing $w$. 	
\begin{algorithm}
	\caption{Frank-Wolfe method with in face directions (FDFW) on a polytope}
	\label{alg:FW}
	\begin{algorithmic}[1]
		\STATE{\textbf{Initialize} $w_0 \in \Q$, $k := 0$}
		
		\STATE{Let $s_k \in \argmax_{y \in \Q}\Sc{\nabla f(w_k)}{y}$ and $d_k^{\mathcal{FW}} := s_k - w_k$.} \label{st:FW}
		\IF{$w_k$ is stationary}
		\STATE{STOP}
		\ENDIF
		\STATE{Let $v_k \in \argmin_{y \in \F(w_k)} \Sc{\nabla f(w_k)}{y}$ and $d_k^{\mathcal{FD}} := w_k - v_k$.}
		\IF{$\Sc{\nabla f(w_k)}{d_k^{\FW}} \geq \Sc{\nabla f(w_k)}{d_k^{\FD}}$}
		\STATE{$d_k := d_k^{\FW}$}
		\ELSE
		\STATE{$d_k := d_k^{\FD}$}\label{st:inface}
		\ENDIF
		\STATE{Choose the stepsize $\alpha_k \in (0, \alpha_{k}^{\max}]$ with a suitable criterion}\label{st:stepsize}
		\STATE{Update: $w_{k+1} := w_k + \alpha_k d_k$}
		\STATE{Set $k : = k+1$. Go to step 2.}
	\end{algorithmic}
\end{algorithm}
The FDFW at every iteration chooses between the classic FW direction $d_k^{\FW}$ calculated at \cref{st:FW} and the in face direction $d_k^{\mathcal{FD}}$ calculated at \cref{st:inface} with the criterion in \cref{st:stepsize}. The classic FW direction points toward the vertex maximizing the scalar product with the current gradient, or equivalently the vertex maximizing the first order approximation $w \mapsto f(w_k) + \Sc{\nabla f(w_k)}{w}$ of the objective $f$. The in face direction $d_k^{\mathcal{FD}}$ is always a feasible direction in $\F(w_k)$ from $w_k$, and it points away from the vertex of the face minimizing the first order approximation of the objective. When the algorithm performs an in face step, we have that the minimal face containing the current iterate either stays the same or its dimension drops by one. The latter case occurs when the method performs a maximal feasible in face step (i.e., a step with $\alpha_k = \alpha_k^{\max}$ and $d_k = d_k^{\FD}$), generating a point on the boundary of the current minimal face. As we will see formally in  \cref{p:FDFWli}, this drop in dimension is what allows the method to quickly identify low dimensional faces containing solutions.\\
We often require the following lower bound on the stepsizes:
\begin{equation} \label{deq:alphalb}
	\alpha_k \geq  \bar{\alpha}_k := \min(\alpha_k^{\max}, \cC\frac{\Sc{\nabla f(w_k)}{d_k}}{\n{d_k}^2}) \tag{S1}
\end{equation}
for some $\cC > 0$. Furthermore, for some convergence results we need the following sufficient increase condition for some constant $\rho >0$:
\begin{equation} \label{eq:deltaflb}
	\fun(w_k + \alpha_k d_k) - \fun(w_k) \geq \rho \bar{\alpha}_k \, \Sc{\nabla \fun(w_k)}{d_k} \tag{S2} \, .
\end{equation}
As explained in the Appendix (see \cref{alphacond}), these conditions generalize properties of exact and Armijo line search. \\

We also define the multiplier functions $\lambda_a$ for $a \in A$, $w\in \R^n$ as
\begin{equation}
	\lambda_a(w) = \Sc{\nabla f(w)}{(a - w)} \, .
\end{equation}
We adapt the well known FW gap  (\cite{bomze2020active}, \cite{lacoste2016convergence}) to the maximization case, thus obtaining the following measure of stationarity
\begin{equation} \label{def:gz}
	G(w) :=  \max_{y \in \Q} \Sc{\nabla f(w)}{(w-y)} = \max_{a \in A} \Sc{\nabla f(w)}{(w - a)} = \max_{a \in A}- \lambda_a(w) \, ,
\end{equation}
as well as an "in face" gap 
\begin{equation} \label{def:gfz}
	G_{\F}(w) = \max (G(w), \max_{b \in \F(w) \cap A} \lambda_b(w)) \, .
\end{equation}
Using these definitions, we have
\begin{equation} \label{eq:gfz}
	\begin{aligned}
		&\Sc{\nabla f(w_k)}{d_k} = \max (\Sc{\nabla f(w_k)}{d_k^{\FW}}, \Sc{\nabla f(w_k)}{d_k^{\FD}}) \\
		= &\max (G(w_k), \max_{y \in \F(w_k)} \Sc{\nabla f(w_k)}{(w_k - y)}) =  G_{\F}(w_k) \, ,	
	\end{aligned}
\end{equation}
where in the second equality we used
\begin{equation} \label{eq:dkfw}
	\begin{aligned}
		\Sc{\nabla f(w_k)}{d_k^{\FW}} = \max_{y \in Q} \Sc{\nabla f(w_k)}{(y - w_k)}	&	
	\end{aligned}
\end{equation} 
and in the third equality  
\begin{equation} \label{eq:fdl}
	\Sc{\nabla f(w_k)}{d_k^{\FD}} = \max_{b \in \F(w_k)} \Sc{\nabla f(w_k)}{(w_k - b)} = \max_{b \in \F(w_k) \cap A} - \lambda_b(w_k) \, . 
\end{equation}
From the definitions it also immediately follows
\begin{equation}
	G_{\F}(w) \geq G(w) \geq 0
\end{equation}
with equality iff $w$ is a stationary point. \\
In order to obtain a local identification result, we need to prove that under certain conditions the method does consecutive maximal in face steps, thus identifying a low dimensional face containing a minimizer. First, in the following lemma we give an upper bound for the maximal feasible stepsize. 
\begin{lem} \label{l:alphakineq}
	If $w_k$ is not stationary, then $\alpha_k \leq G(w_k) / G_{\F}(w_k)$. 
\end{lem}
\begin{proof}
	Notice that since $w_k$ is not stationary we have $G(w_k) > 0$ and therefore also $G_{\F}(w_k) > 0$. Now 
	\begin{equation*}
		\Sc{\nabla f(w_k)}{(w_k + \alpha_kd_k)} \leq \max_{y \in \Q} \Sc{\nabla f(w_k)}{y} = 	\Sc{\nabla f(w_k)}{(w_k +  d_k^{\FW})} = \Sc{\nabla f(w_k)}{w_k} + G(w_k) \, ,
	\end{equation*} 
	where in the inequality we used $w_k + \alpha_k d_k \in \Q$. Subtracting $\Sc{\nabla f(w_k)}{w_k}$ on both sides we obtain
	\begin{equation}
		\alpha_k \Sc{\nabla f(w_k)}{d_k} \leq G(w_k) \, .
	\end{equation}
	and the thesis follows by applying \cref{eq:gfz} to the LHS.
\end{proof}
We can now prove a local identification result.
\begin{prop}[FDFW local identification] \label{p:FDFWli}
	Let $p$ be a stationary point for $f$ defined on $\Q$ and assume that \cref{deq:alphalb} holds. We have the following properties:
	\begin{itemize}
		\item[(a)]there exists $r^*(p) > 0$ such that if $w_{k} \in B_{r^*(p)}(p) \cap \F_e(\nabla f(p))$ then $w_{k + 1} \in \F_e(\nabla f(p))$;
		\item[(b)] for any $\delta > 0$ there exists $r(\delta, p) \leq \delta$ such that if $w_{k} \in B_{r(\delta, p)}(p)$ then  $w_{k + j} \in \F_e(\nabla f(p)) \cap B_{\delta}(p)$ for some $j \leq \dim (\F(w_{k}))$.
	\end{itemize} 
\end{prop}
\begin{proof}
	(a) Notice that by definition of exposed face and stationarity conditions
	\begin{equation} \label{eq:lam}
		\lambda_a(p) \leq 0
	\end{equation}
	for every $a \in A$, with equality iff $a \in \F_e(\nabla f(p))$. Then by continuity we can take $r^*(p)$ small enough so that $\lambda_a(w) < 0$ for every $a \in A \sm (A \cap \F_e(\nabla f(p)))$. Under this condition, if $w_k \in B_{r^*(p)}(p)$ then the method cannot select a FW direction pointing toward an atom outside the exposed face $\F_e(\nabla f(p))$, because all the atoms maximizing the RHS of \cref{def:gz} must necessarily be in $\F_e(\nabla f(p))$. In particular if $w_k \in  B_{r^*(p)}(p)\cap \F_e(\nabla f(p))$ then the method selects either an in face direction or a FW direction pointing toward a vertex in $\F_e(\nabla f(p))$. In both cases, $w_{ k + 1} \in \F_e(\nabla f(p))$. \\
	\newcommand{\rp}[1]{r^{(#1)}(\delta, p) }
	(b) Let $D$ be the diameter of $\Q$. We now consider $r^{(0)}(\delta, p) \leq \min(\delta, r^*(p))$ such that, for every $w \in B_{r^{(0)}(\delta, p)}(p)$
	\begin{equation} \label{eq:c1}
		\max_{a \in A} \lambda_a(w) < \min_{b \in A \sm  \F_e(\nabla f(p))}\min(-\lambda_b(w), \frac{\cC}{D^2} \lambda_b(w)^2) 	\, .	
	\end{equation}
	As we will see in the rest of the proof this upper bound together with \cref{l:alphakineq} ensures in particular that the FDFW performs maximal in face steps in $B_{r^{(0)}(\delta, p)}(p) \sm \F_e(\nabla f(p))$.	Furthermore, \cref{eq:c1} can always be satisfied thanks to \cref{eq:lam} and by the continuity of multipliers. We then define recursively for $1\leq l \leq n$ a sequence $r^{(l)}(\delta, p) \leq \rp{l - 1}$ of radii small enough so that, for 
	\begin{equation} \label{eq:Ml}
		M_l = \sup_{w \in B_{(l)}(p) \sm \F_e(\nabla f(p))} G(w)/G_{\F}(w) \, ,	
	\end{equation}
	with $B_{(l)}(p) := B_{{r^{(l)}}(\delta, p)}(p)$ we have
	\begin{equation} \label{eq:c2}
		r^{(l)}(\delta, p) + DM_l < r^{(l - 1)}(\delta, p) \, .
	\end{equation}
	Again this sequence can always be defined thanks to the continuity of multipliers. Finally, we define $r(\delta, p) = \rp{n}$. \\
	Given these definitions, when $w_k \in B_{(l)}(p) \subset B_{(0)}(p) $ and $w_k $ is not in $ \F_e(\nabla f(p))$ an in face direction is selected, because
	\begin{equation}
		\Sc{\nabla f(w_k)}{d_k^{\FW}} = \max_{a \in A} \lambda_a(w) < \min_{b \in A \sm  \F_e(\nabla f(p))} - \lambda_b(w) \leq  \max_{b \in \F(w_k) \cap A} - \lambda_b(w_k) = \Sc{\nabla f(x_k)}{d_k^{\FD}} \, ,
	\end{equation} 
	where we used \cref{eq:c1} in the first inequality, $w_k \notin \F_{e}(p)$ in the second, and \cref{eq:fdl} in the second equality. 
	We now want to prove that in this case $\alpha_k$ is maximal reasoning by contradiction. On the one hand, we have
	\begin{equation} \label{eq:algl}
		\alpha_k \geq \cC \frac{\Sc{\nabla f(x_k)}{d_k}}{\n{d_k}^2} \geq \frac{\cC}{D^2} \Sc{\nabla f(x_k)}{d_k} = \frac{\cC}{D^2} G_{\F}(w_k)
	\end{equation}
	where we used the assumption \cref{deq:alphalb} in the first inequality, $\n{d_k} \leq D$ in the second and $G_{\F}(w_k) =\Sc{\nabla f(x_k)}{d_k^{\FD}} $ together with $d_k = d_k^{\FD}$ in the last one. \\
	On the other hand,  
	\begin{equation}\label{eq:inegzk}
		\begin{aligned}
			G(w_k) = \max_{a \in A} \lambda_a(w_k) < &  \frac{\cC}{D^2} \min_{b \in A \sm  \F_e(\nabla f(p))} \lambda_b(w)^2  \leq  \frac{\cC}{D^2} \max_{b \in \F(w_k)} \lambda_b(w)^2 \\ 
			= & \frac{\cC}{D^2}(\Sc{\nabla f(w_k)}{d_k})^2 = \frac{\cC}{D^2} G_{\F}(w_k)^2
		\end{aligned}
	\end{equation}
	where we used \cref{eq:c1} in the first inequality, $w_k \notin  \F_e(\nabla f(p))$ in the second, \cref{eq:fdl} together with $d_k = d_k^{\FD}$ in the second equality, and \cref{eq:gfz} in the third equality. \\
	The inequality \cref{eq:inegzk} leads us to a contradiction with the lower bound on $\alpha_k$ given by \cref{eq:algl}, since it implies
	\begin{equation} \label{eq:alphakup}
		\alpha_k \leq \frac{G(w_k)}{G_{\F}(w_k)} < \frac{\cC}{D^2} G_{\F}(w_k) \, ,
	\end{equation}
	where we applied \cref{l:alphakineq} in the first inequality and \cref{eq:inegzk} in the second.
	\\ 
	Assume now $w_k \in B_{(n)}(p)$. We prove by induction that, for every $j \in [-1 : \tx{dim}(\F(w_k)) - 1]$, if $\{w_{k + i}\}_{0 \leq i \leq  j} \cap \F_e(\nabla f(p)) = \emptyset$  then $w_{k + j + 1} \in B_{(n - j - 1)}(p)$. 
	For $j = - 1$ we have $w_k \in B_{(n)}(p)$ by assumption. Now if $\{w_{k + i}\}_{0 \leq i \leq j} \cap \F_e(\nabla f(p)) = \emptyset$ we have
	\begin{equation}
		\begin{aligned}
			&	\n{w_{k + j + 1} - p} \leq \n{w_{k + j} - p} + \n{w_{k + j + 1} - w_{k + j}} < \rp{n - j}  +  \n{w_{k + j + 1} - w_{k + j}}  \\
			= & \rp{n - j} + \alpha_k \n{d_k}	\leq   \rp{n - j} +  D\frac{G(w_k)}{G_{\F}(w_k)}	\leq   \rp{n - j} + D M_{n - j} < 	\rp{n - j - 1} \, ,
		\end{aligned}
	\end{equation}
	where we used the inductive hypothesis $w_{k + j} \in B_{(n - j)}(p)$ in the second inequality, \cref{l:alphakineq} in the third inequality,  \cref{eq:Ml} in the fourth inequality and the assumption \cref{eq:c2} in the last one.  In particular $w_{k + j + 1} \in B_{(n - j - 1)}(p)$ and the induction is completed. \\
	Since $B_{(n - j)}(p) \subset B_{(0)}(p)$, if  $w_{k + j} \in (B_{(n - j)}(p) \sm \F_e(\nabla f(p))$ then $\alpha_{k + j}$ must be maximal and therefore $\tx{dim}(\F(w_{k + j + 1})) < \tx{dim}(\F(w_{k + j}))$. But starting from the index $k$ the dimension of the current face can decrease at most $\tx{dim}(\F(w_{k})) < n$ times in consecutive steps, so there must exists $j \in [0, \tx{dim}(\F(w_k))]$ such that $w_{k + j} \in \F_{e}(\nabla f(p))$. Taking the minimum $j$ satisfying this condition we also obtain $w_{k + j} \in B_{(0)}(p) \subset B_{\delta}(p)$.
\end{proof}
\newcommand{\rs}{r^{(s)}}
A straightforward adaptation of results from \cite{bomze2020active} implies convergence to the set of stationary points for the FDFW.
\begin{prop} \label{p:fdfwconv}		
	If \cref{deq:alphalb} and \cref{eq:deltaflb} hold, then all the limit points of the FDFW are contained in the set of stationary points of $f$.
\end{prop}
\begin{proof}
	The proof presented in the special case of the simplex in \cite{bomze2020active}, where the FDFW coincides with the away-step Frank-Wolfe, extends to generic polytopes in a straightforward way. 
\end{proof}
\newcommand{\tr}{\tilde{r}(p)}
In the next lemma we improve the FDFW local identification result given in \cref{p:FDFWli} under an additional strong concavity assumption for the face containing the solution, satisfied in particular by $h_{\G}$.
\begin{lem} \label{eq:scv} 
	Let $p$ be a stationary point for $f$ restricted to $\Q$. Assume that \cref{deq:alphalb} holds and that $f$ is strongly concave\footnote{in fact, we only need strict concavity of $f$ here.} in $\F_e(\nabla f(p))$. Then, for a neighborhood $U(p)$ of $p$, if $w_0 \in U(p)$:
	\begin{itemize}
		\item[(a)] if $\{f(w_k)\}$ is increasing, there exists $k \in [0\! :\! \tx{dim}(\F(w_0))]$ such that $w_{k + i} \in \F_e(\nabla f(p))$ for every $i \geq 0$;
		\item[(b)] if in addition also \cref{eq:deltaflb} holds, then  $\{w_{k + i}\}_{i \geq 0}$ converges to $p$. 		
	\end{itemize} 
\end{lem}
\begin{proof}
	(a)	Let $\mu$ be the strong concavity constant of $f$ restricted to $\F_e(\nabla f(p))$, so that 
	\begin{equation}\label{eq:Ff}
		f(w) \leq f(p) - \frac{\mu}{2}\n{w - p}^2
	\end{equation}
	for every $w \in \F_e(\nabla f(p))$. For $\varepsilon = \frac{\mu r^{*}(p)^2}{2}$, let $\mathcal{L}_{\varepsilon}$ be the superlevel of $f$ for $f(p) - \varepsilon$:
	\begin{equation}
		\mathcal{L}_{\varepsilon} = \{y \in \Q \ | \ f(y) > f(p) - \varepsilon\} \, .
	\end{equation}
	Let now $\bar{r} = r(\delta, p)$ defined as in \cref{p:FDFWli}, with $\delta = r^*(p)$. By \cref{eq:Ff} it follows $\mathcal{L}_{\varepsilon} \cap \F_e(\nabla f(p)) \subset B_{r^*(p)}(p)$. Assume now $w_0 \in U(p)$ with $U(p) =B_{\bar{r}}(p) \cap \mathcal{L}_{\varepsilon}$. By applying \cref{p:FDFWli} we obtain that there exists $k \in [0\! : \!\tx{dim}(\F(w_0))]$ such that $w_{k}$ is in $ \F_e(\nabla f(p))\cap B_{r^*(p)}(p)$. But since $f(w_k) \geq f(w_0) > f(p) - \varepsilon$ we have the stronger condition $w_k \in \mathcal{L}_{\varepsilon} \cap  \F_e(\nabla f(p))$.To conclude, notice that the sequence cannot escape from this set, because for $i \geq 0$ $w_{k + i} \in \mathcal{L}_{\varepsilon}$ implies that also $w_{k + i + 1}$ is in $\mathcal{L}_{\varepsilon}$, and $w_{k + i} \in \mathcal{L}_{\varepsilon} \cap  \F_e(\nabla f(p)) \subset B_{r^*(p)}(p)\cap  \F_e(\nabla f(p)) $ implies that also $w_{k + i + 1}$ is in $ \F_e(\nabla f(p))$.    \\
	(b) By point (a) $\{w_{k + i}\}_{i \geq 0}$ is contained in $ \F_e(\nabla f(p))$.  But by assumption $f$ is strongly concave in $ \F_e(\nabla f(p))$ with $p$ global maximum and the only stationary point. To conclude it suffices to apply \cref{p:fdfwconv}.
\end{proof}
\begin{cor} \label{cor:fwglob}
	Let $\{w_k\}$ be a sequence generated by the FDFW and assume that at least one limit point $p$ is stationary and such that $f$ is strongly concave in $\F_e(\nabla f(p))$. Then under the conditions \cref{deq:alphalb} and \cref{eq:deltaflb} on the stepsize we have $w_k \rightarrow p $ with $w_k \in \F_e(\nabla f(p))$ for $k$ large enough.
\end{cor}
\begin{proof}
	Follows from \cref{eq:scv} by observing that the sequence must be ultimately contained in $U(p)$. 
\end{proof}
We can now prove local convergence and identification for the FDFW applied to the $s - $defective maximal clique formulation \cref{deq:P}.
\begin{prop}[FDFW local convergence] \label{p:fdfwlocal}
	Let $p = (x^{(C)}, y^{(p)}) \in \M_s(\G)$, let $\{z_k\}$ be a sequence generated by the FDFW. Then under \cref{deq:alphalb} there exists a neighborhood $U(p)$ of $p$ such that if $\bar{k} := \min \{k \in \N_{0} \ | \ z_k \in U(p) \}$ we have the following properties:
	\begin{itemize}
		\item[(a)] if $h_{\G}(z_k)$ is monotonically increasing, then $\supp(z_{k}) = C$ and $y_k = y^{(p)}$ for every $k \geq \bar{k} + \dim \F(w_{\bar{k}})$;
		\item[(b)]  if \cref{eq:deltaflb} also holds then $z_{k} \rightarrow p$.
	\end{itemize}
\end{prop}
\begin{proof}
	Let $A(p) = A_{\G} + A(y^{(p)})$. Then for $x \in \Delta^{(C)}$
	\begin{equation} \label{eq:cdelta}
		\begin{aligned}
			& x^{\intercal}A(p)x = \sum_{(i, j) \in V^2} x_i A(p)_{ij} x_j = \sum_{i \in C} x_i (\sum_{j \in C} A(p)_{ij} x_j)  =  \sum_{i \in C} x_i (\sum_{j \in C \sm \{i\} } x_j) \\
			= & \sum_{i \in C} (x_i \sum_{j \in C} x_j - x_i^2) =  (\sum_{i \in C} x_i)^2 - \sum_{i \in C} x_i^2 \, ,
		\end{aligned}
	\end{equation}
	where in the first equality we used $\tx{supp}(x) = C$, in the second that $C$ is a clique n $G(y^{(p)})$, and $\sum_{i \in C}x_i = \sum_{i \in V} x_i = 1$. \\
	Observe now that the function $x \mapsto h_{\G}(x, y^{(p)})$ is strongly concave in $\Delta^{(C)}$. Indeed for $x \in \Delta^{(C)}$ 
	\begin{equation} \label{eq:sc}
		\begin{aligned}
			h_{\G}(x, y^{(p)}) = x^{\intercal}A(p)x + \frac{\alpha}{2}\n{x}^2 + \frac{\beta}{2}\n{y^{(p)}}^2  = & (\sum_{i \in C} x_i)^2 - \sum_{i \in C} x_i^2 +  \frac{\alpha}{2}\n{x}^2 + \frac{\beta}{2}\n{y^{(p)}}^2 \\
			= & 1 - (1 - \frac{\alpha}{2})\sum_{i \in C} x_i^2 +  \frac{\beta}{2}\n{y^{(p)}}^2 \, ,	
		\end{aligned}
	\end{equation}
	where in the second equality we used \cref{eq:cdelta}. The RHS of \cref{eq:sc} is strongly concave in $x$ since $\alpha \in (0, 2)$ so that $- (1 - \alpha/2) \in (- 1, 0)$. This together with \cref{l:sc} gives us the necessary assumptions to apply \cref{eq:scv}. 
\end{proof}
As a corollary, we have the following global convergence result under the mild assumption that the set of limit points contains one local minimizer.
\begin{cor}[FDFW global convergence]
	Let $\{z_k\}$ be a sequence generated by the FDFW and assume that at least one limit point $p = (x^{(C)}, y^{(p)})$ of $\{z_k\}$ is in $\M_s(\G)$. Then under the conditions \cref{deq:alphalb} and \cref{eq:deltaflb} on the stepsize we have $z_k \rightarrow p $ with $\supp(x_k) \subset C$ and $y_k = y_p$ for $k$ large enough.
\end{cor}
\begin{proof}
	Follows from \cref{cor:fwglob} where all the necessary assumptions are satisfied as for \cref{p:fdfwlocal}.
\end{proof}

\begin{rem}
	When the sequence converges to a first order stationary point $z^* = (x^*, y^*)$ which is not a local maximizer, one can use the procedure described in \cite{stozhkovcontinuous} to obtain an $s$-defective clique $C$ and $y \in D_s(\G)$ with $h_{\G}((x^{(C)}, y)) > h(z^*)$. The cost of the procedure is $O(|\tx{supp}(x^*)|^2)$.
\end{rem}
\section{A Frank-Wolfe variant for $s$-defective clique}\label{Sec:FWvar}
As can be seen from numerical results, one drawback of the standard FDFW applied to the $s$-defective clique  formulation \cref{deq:P} is the slow convergence of the high dimensional $y$ component. Since this component is "tied" to the $x$ component, it is not possible to speed up the convergence by changing the regularization term without compromising the quality of the solution.  Motivated by this challenge, we introduce a tailored Frank-Wolfe variant, namely FWdc,  for the maximum $s$-defective clique formulation \cref{deq:P}, which exploits the product domain structure of the problem at hand by employing separate updating rules for the two blocks.   \\
\setcounter{ALC@unique}{0}	
\begin{algorithm}
	\caption{Frank-Wolfe variant for $s$-defective clique}
	\label{alg:BCFW}
	\begin{algorithmic}[1]
		\STATE{\textbf{Initialize} $z_0:=(x_0, y_0) \in \p$, $k := 0$}
		\IF{$z_k$ is stationary}
		\STATE{STOP}
		\ENDIF
		\STATE{Compute $x_{k + 1}$ applying one iterate of \cref{alg:FW} with $w_0 = x_k$ and $f(w)=  h_{\G}(w, y_k)$.} \label{st:xk1}
		\STATE{Let $y_{k+1} \in \argmax_{y \in D'_s(\G)} \Sc{\nabla_y h_{\G}(x_k, y_k)}{y} $.} \label{st:yk1}
		\STATE{Set $k : = k+1$. Go to step 2.}
	\end{algorithmic}
\end{algorithm}
In particular, at every iteration the method alternates a FDFW step on the $x$ variables (\cref{st:xk1}) with a full FW step on the $y$ variable (\cref{st:yk1}), so that $y_k$ is always chosen in the set of vertices $D_s(\G)$ of $D'_s(\G)$. Furthermore, as we prove in the next proposition, $\{y_k\}$ is ultimately constant. This allows us to obtain convergence results by applying the general properties of the FDFW proved in the previous section to the $x$ component. 
\begin{prop} \label{p:algo2}
	In the FWdc if $h_{\G}(z_k)$ is increasing then $\{y_k\}$ can change at most $\frac{2}{\beta} - \frac{2 - \alpha}{\beta |C^*|} + s$ times, with $C^*$ $s$-defective clique of maximal cardinality.  
\end{prop}
\begin{proof}
	Assume that $y_k$ and $y_{k + 1}$ are distinct vertices of $D_s'(\G)$. Then 
	\begin{equation}
		\begin{aligned}
			&	h_{\G}(z_{k + 1}) - h_{\G}(z_k) \geq \Sc{\nabla h_{\G}(z_k)}{(z_{k + 1} - z_k)} + \frac{\beta}{2}\n{z_{k + 1} - z_k}^2 \\  = &\Sc{\nabla_y h_{\G}(z_k)}{(y_{k + 1} - y_k)} + \frac{\beta}{2} \n{y_{k + 1} - y_k}^2 \geq \frac{\beta}{2} > 0 \,  	
		\end{aligned}
	\end{equation}
	where we used the $\beta - $strong convexity of $y \mapsto h_{\G}(x , y)$ in the first inequality, $x_k = x_{k + 1}$ in the equality, $y_{k + 1} \in \argmax_{y \in \p} \Sc{\nabla_y h_{\G}(z_k)}{y}$ and the fact that the distance between vertices of $D_s'(\G)$ is at least 1 in the second inequality. \\
	Therefore $y_k$ can change at most
	\begin{equation*}
		\max_{z \in \p} \frac{2(h_{\G}(z) - h_{\G}(z_0))}{\beta} \leq \max_{z \in \p} \frac{2h_{\G}(z)}{\beta} = \frac{1 - 1/|C^*| + \alpha/2|C^*| + s\beta/2 }{\beta/2} = \frac{2}{\beta} + \frac{\alpha - 2}{\beta|C^*|} + s
	\end{equation*}
	times, where we used $h_{\G} \geq 0$ in the first inequality, and \cref{c:globalm} in the second inequality. 
\end{proof}
\begin{cor} \label{cor:algo2conv}
	Let $\{z_k\}$ be a sequence generated by Algorithm 2. 
	\begin{itemize}
		\item[1.]  If conditions \cref{deq:alphalb} and \cref{eq:deltaflb} hold on the stepsizes, then $\{z_k\}$ converges to the set of stationary points.
		\item[2. ]If the stepsize is given by exact line search or Armijo line search and the set of limit points of $\{z_k\}$ is finite, then $z_k \rightarrow p$ with $p$ stationary. 
	\end{itemize} 
\end{cor}
\begin{proof}
	As a corollary of \cref{p:algo2}, an application of Algorithm 2 reduces, after a finite number of changes for the variable $y$, to an application of the FDFW on the simplex for the optimization of a quadratic objective. After noticing that on the simplex the FDFW coincides with the AFW, point 1 follows directly from \cref{p:fdfwconv}, and point 2 follows from \cite[Theorem 4.5]{bomze2020active}. 
\end{proof}
For a clique $C$ of $\G(y)$ different from $\G$ we define $m(C, \G(y))$ as 
\begin{equation}
	\min_{v \in V \sm C} |C| - |E^y(v) \cap C| \, ,
\end{equation}
that is the minimum number of edges needed to increase by 1 the size of the clique. \\
We now give an explicit bound on how close the sequence $\{x_k\}$ generated by \cref{alg:BCFW} must be to $x^{(C)}$ for the identification to happen.
\newcommand{\lmax}{\delta_{\max}}
\begin{prop}
	Let $\{z_k\}$ be a sequence generated by  \cref{alg:BCFW}, $\bar{y} \in D^s(\G) $, $C$ be a clique in $\G(\bar{y})$,  let $\lmax$ the maximum eigenvalue of the adjacency matrix $\bar{A} : = A_{\G} + A(\bar{y})$.
	Let $\bar{k}$ be a fixed index in $\N_0$, $I^c$ the components of $\tx{supp}(x_{\bar{k}})$ with index not in $C$ and let $L := 2 \delta_{\max} + \alpha$. Assume that $y_{\bar{k} + j} = \bar{y} $ is constant for $0 \leq j \leq |I^c|$, that \cref{deq:alphalb} holds for $\cC = 1/L$, and that 
	\begin{equation} \label{eq:sdc}
		\n{x_{\bar{k}} - x^{(C)}}_1 \leq  \frac{m_{\alpha}(C, \G(y_{\bar{k}}))}{m_{\alpha}(C,  \G(y_{\bar{k}})) + 2|C|\lmax + |C|\alpha }
	\end{equation}
	for $m_{\alpha}(C,  \G(y_{\bar{k}})) = m(C,  \G(y_{\bar{k}})) - 1 + \alpha/2$.
	Then $\tx{supp}(x_{\bar{k} + |I^c|})  = C$.
\end{prop}
\begin{proof}
	Since $y_k$ does not change for $k \in [\bar{k}\! :\! \bar{k} + |I^c|]$, Algorithm 2 corresponds to an application of the AFW to the simplex $\Delta_{n - 1}$ on the variable $x$. For $1 \leq i \leq n$ let $\lambda_{i}(x) = \frac{\partial}{\partial x_i} h_{\G}(x, y_{\bar{k}})$ be the multiplier functions associated to the vertices of the simplex, and let 
	\begin{equation}
		\lambda_{\min} = \min_{i \in V \sm C} - \lambda_i(x^{(C)}) \, ,
	\end{equation}
	be the smallest negative multiplier with corresponding index not in $C$. Let $L'$ be a Lipschitz constant for $\nabla_x h_{\G}(x, y)$ with respect to the variable $x$. By \cite[Theorem 3.3]{bomze2020active} if 
	\begin{equation} \label{deq:bomze}
		\n{x_{\bar{k}} - x^{(C)}}_1 < \frac{\lambda_{\min}}{\lambda_{\min} + 2L'} 
	\end{equation}
	we have the desired identification result. \\		
	We now prove that we can take $L'$ equal to $L$ in the following way:
	\begin{equation}
		\n{\nabla_x h_{\G}(x', y_{\bar{k}}) - \nabla_x h_{\G}(x, y_{\bar{k}})} = \n{2\bar{A}(x' - x) + \alpha(x' - x)} \leq (2 \lmax + \alpha) \n{x' - x} \, ,
	\end{equation}
	where we used $\nabla_x h_{\G}(x, y) = 2\bar{A}x + \alpha x$ in the equality. As for the multipliers, for $i \in V \sm C$ we have the lower bound
	\begin{equation} \label{lambdaineq}
		\begin{aligned}
			-	\lambda_i(x^{(C)}) = \Sc{\nabla_x h_{\G}(x^{(C)}, y_{\bar{k}})}{(x^{(C)} - e_i)}  = & \frac{- 2|C\cap E^{y_{\bar{k}}}(i)| + 2|C| -  2 + \alpha}{|C|} \\
			\geq & \frac{2m_{\alpha}(C, \G(y_{\bar{k}}))}{|C|} 
		\end{aligned}
	\end{equation}
	by combining  \cref{eq:x1piece} and \cref{eq:x2piece} in the second equation. We can now bound $\lambda_{\min}$ from below:
	\begin{equation} \label{deq:lm}
		\lambda_{\min} = \min_{i \in V \sm C} - \lambda_i(x^{(C)}) \geq \min_{i \in V \sm C}\frac{2|C| - 2|C\cap E^y_{\bar{k}}(i)| - 2 + \alpha}{|C|} \geq  \frac{2m_{\alpha}(C, \G(y_{\bar{k}}))}{|C|} \, ,
	\end{equation}
	where we applied \cref{lambdaineq} in the inequality. 
	Finally, we have
	\begin{equation} \label{deq:ri}
		\frac{\lambda_{\min}}{\lambda_{\min} + 2L} \leq  \frac{m_{\alpha}(C, \G(y_{\bar{k}}))}{m_{\alpha}(C, \G(y_{\bar{k}})) + 2|C|\lmax + |C|\alpha} \, 
	\end{equation}
	where we applied \cref{lambdaineq} together with \cref{deq:lm} in the inequality. The thesis follows applying \cref{deq:ri} to the RHS of \cref{deq:bomze}.
\end{proof}
\begin{rem} \label{r:dmax}
	It is a well known result that for any graph the maximal eigenvalue $\delta_{\max}$  of the adjacency matrix is less than or equal to $d_{\max}$, the maximum degree of a node (see, e.g., \cite{cvetkovic1990largest}). Then condition \cref{eq:sdc} can be replaced by
	\begin{equation}
		\n{x_{\bar{k}} - x^{(C)}}_1 \leq  \frac{m_{\alpha}(C, \G(y_{\bar{k}}))}{m_{\alpha}(C, \G(y_{\bar{k}})) + 2|C|d_{\max} + |C|\alpha} \, .
	\end{equation}
\end{rem}
\section{Numerical results}\label{Sec:NumRes}
In this section we report on a numerical comparison of the methods. We remark that, even though these methods only find maximal $s$-defective cliques, they can still be applied as a heuristic to derive lower bounds on the maximum $s$-defective clique within a global optimization scheme.	
With our tests, we aim to achieve the followings:
\begin{itemize}
	\item empirically verify the active set identification property of the proposed methods;
	\item prove that the proposed FW variant  is faster than the FDFW on these problems, while mantaining the same solution quality;
	\item make a preliminary comparison between the FW methods and the CONOPT solver used in \cite{stozhkovcontinuous}. 
\end{itemize}

In the tests, the regularization parameters were set to $\alpha = 1$ and $\beta = 2/n^2$. An intuitive motivation for this choice of $\beta$ can be given by imposing that the missing edges for an identified $s$-defective clique are always included in the support of the  FW vertex. Formally, if $x_k = x^{(C)}$ with $C$ an $s$-defective clique and $(y_k)_{ij} = 0$ with $\{i, j\} \in {C \choose 2}$ we want to ensure that the FW vertex $s_k = (x^{(s_k)}, y^{(s_k)})$ is such that $y^{(s_k)}_{ij} = 1$. Now for $\{l, m\} \notin {C \choose 2}$ and assuming $|C| < n$ (otherwise $C = V$ and the problem is trivial) we have
\begin{equation}\label{ppyh}
	\frac{\partial}{\partial y_{ij}} h_{\G}(x_k, y_k) =	\frac{2}{|C|^2} > \frac{2}{n^2} = \beta \geq  \frac{\partial}{\partial y_{lm}} h_{\G}(x_k, y_k)
\end{equation}
where the first equality and the last inequality easily follow from \cref{eq:dph}. From \cref{ppyh} it is then immediate to conclude that $\{i, j\}$ must be in the support of $y^{(s_k)}$. \\
We used the stepsize $\alpha_{k} = \bar{\alpha}_k$  with $\bar{\alpha}_k$ given by \cref{deq:alphalb} for $\cC= 2$, corresponding to an estimate of $0.5$ for the Lipschitz constant $L$ of $\nabla h_{\G}$. A gradient recycling scheme was adopted to use first order information more efficiently (see \cite{rinaldi2020avoiding} for details).  The code was written in MATLAB and the tests were performed on an Intel Core i7-10750H CPU 2.60GHz, 16GB RAM. \\
The 50 graph instances we used in the tests are taken from the Second DIMACS Implementation Challenge \cite{johnson1993cliques}. These graphs are a common benchmark to assess the performance of algorithms for maximum (defective) clique problems (see references in \cite{stozhkovcontinuous}), and the particular instances we selected coincide with the ones employed in \cite{stozhkovcontinuous} in order to ensure a fair comparison at least for the quality of the solutions. Following the rule adopted in \cite{stozhkovcontinuous}, for each triple $(\G, s, \mathcal{A})$ with $\G$ a graph from the 50 instances considered, $s \in [1\! : \! 4]$, $\mathcal{A}$ the FDFW or the FWdc, we set a global time limit of 600 seconds and employed a simple restarting scheme with up to 100 random starting points. The algorithms always completed 100 runs within the time limit, with the exception of 2 instances (see \cref{tab:2}). For both algorithms the $x$ component of the starting point was generated with MATLAB's function rand and then normalized dividing it by its sum. An analogous rule was applied to generate the $y$ component for the starting point of the FDFW, while for the FWdc the $y$ component was simply initialized to 0. For the stopping criterion, two conditions are required: the current support of the $x$ components coincides with an $s$-defective clique, and the FW gap is less than or equal to $\varepsilon:= 10^{-3}$.
In the experiments, both algorithms always terminated having identified an $s$-defective clique, thus providing an empirical verification of the results we proved in this paper.  \\
\begin{figure}[t]
	\centering
	\includegraphics[width=0.6\textwidth]{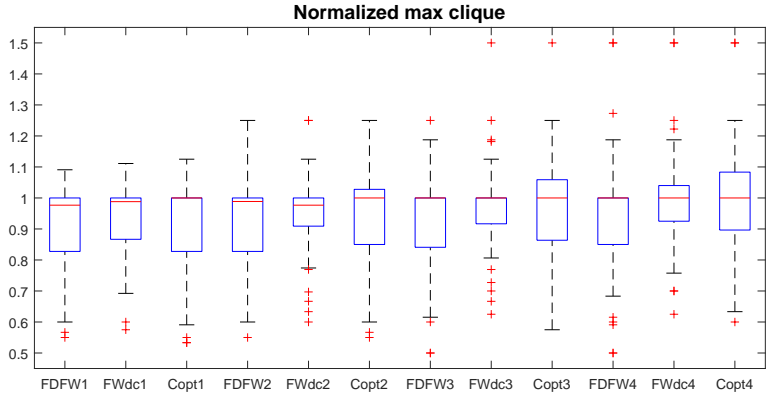}
	\caption{$\mathcal{A} i$ is the box plot of the maximum clique found within the 600 seconds/ 100 runs limit for each instance by the method $\mathcal{A}$ for $s = i$ divided by the maximum cardinality clique of the instance. }
	\label{fig:1}
\end{figure}
\begin{figure}[t]
	\centering
	\includegraphics[width=0.6\textwidth]{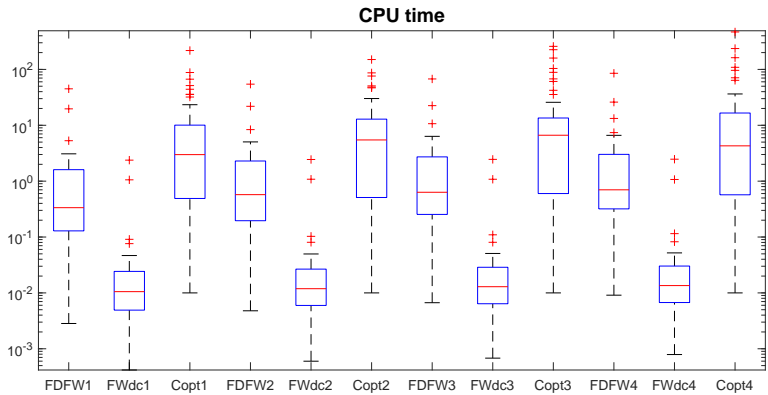}
	\caption{$\mathcal{A} i$ is the box plot of the average running time for each instance for the method $\mathcal{A}$ and $s = i$.}
	\label{fig:2}
\end{figure}

In the boxplots, each series consists of 50 values corresponding to aggregate data for the runs performed on the 50 instances. The data for the CONOPT solver are taken from \cite{stozhkovcontinuous}. The red lines represent the median of the values in each series, and the boxes extend from the 25th percentile $q_1$  of the observed data to the 75th percentile $q_3$. The whiskers cover all the other values in a range of $[q_1 - w (q_3 - q_1), q_3 + w(q_3 - q_1)]$, with the coefficient $w$ equal to $2.7$ times the standard deviation of the values. \\
In \cref{fig:1}, the bar $\mathcal{A}i$ represents the distribution of the maximum cardinality of the $s$-defective clique found by method $\mathcal{A}$ with $s = i$, divided by the maximum clique cardinality of the instance. Notice that some data points are greater than 1, as expected since for $s > 0$ the cardinality of an $s$-defective clique can exceed the maximum clique cardinality. While the variance is higher for the max cliques found by the CONOPT solver, no significant differences can be seen for the median size of max cliques. \\
In \cref{fig:2}, $\mathcal{A}i$ represents the distribution of average running times in seconds (on a logarithmic scale, explainng the asymmetry of the box plots) of method $\mathcal{A}$ for $s = i$. Here we can see that FWdc outperforms FDFW by about 1.5 orders of magnitude, which in turns outperforms the CONOPT solver by about 1 order of magnitude. This is even more impressive if we take into account the fact that the CONOPT solver is written in C++, while FW variants are written in MATLAB. \\

\section{Appendix}
\subsection{Line searches} \label{sa:ls}
Here we briefly report some relevant results about line searches proved in \cite{bomze2020active} showing a connection between well known line searches and conditions \cref{deq:alphalb}, \cref{eq:deltaflb}. \\
Recall that the exact line search stepsize is given by
\begin{equation} \label{eq:ls}
	\alpha_k \in \argmax_{\alpha \in [0, \alpha_{k}^{\max}]} f(w_k + \alpha d_k) \, .
\end{equation}
The stronger condition $\alpha_k = \max \{ \argmax_{\alpha \in [0, \alpha_{k}^{\max}]} f(w_k + \alpha d_k) \}$	is required for \cite[Theorem 4.3]{bomze2019first}, used to prove \cref{cor:algo2conv},   \\
The stepsize $\alpha_k$ given by the Armijo line search always satisfies the condition
\begin{equation} \label{Armijo}
	f(w_k + \alpha_k d_k) - f(w_k) \geq c_1 \alpha_k \Sc{\nabla f(w_k)}{d_k} \, ,
\end{equation}
for some constant $c_1 \in (0, 1)$. This stepsize is produced by considering a sequence $\{\beta_k^{(j)}\}_{j \in \mathbb{N}_0}$ of tentative stepsizes given by $\beta_k^{(0)} = \alpha_{k}^{\max}$, $\beta_{k}^{(j + 1)} = \gamma \beta_{k}^{(j)}$, with $\gamma \in (0, 1)$, and taking the largest tentative stepsize satisfying \cref{Armijo}. \\ 
We report here for completeness results from \cite{bomze2020active} proving that Armijo and exact line search satisfy conditions \cref{deq:alphalb} and \cref{eq:deltaflb}.
\begin{lem}\label{alphacond}
	Consider a sequence $\{w_k\}$ in $\Q$ such that $w_{k  + 1} = w_{k} + \alpha_k d_k$ with $\alpha_k \in \R_{\geq 0}$, $d_k\in \R^n$. Assume that $d_k$ is a proper ascent direction in $w_k$, i.e. $\Sc{\nabla f(w_k)}{d_k} > 0$. 
	\begin{itemize}
		\item[1.] If $\alpha_k$ is given by exact line search, then \cref{deq:alphalb} and \cref{eq:deltaflb} are satisfied with $\cC = \frac{1}{L}$ and $\rho = \frac{1}{2}$. 
		\item[2.] If $\alpha_k$ is given by the Armijo line search described above, then \cref{deq:alphalb}  and \cref{eq:deltaflb} are satisfied with $\cC= \frac{2\gamma(1 - c_1)}{L}$ and $\rho = c_1\min\{1,2\gamma (1- c_1)\}<1.$ 
	\end{itemize}
\end{lem}
\begin{proof}
	Point 1 follows from \cite[Lemma B.1]{bomze2020active} and point 2 follows from \cite[Lemma B.3]{bomze2020active}.
\end{proof}
\clearpage
\subsection{Detailed numerical results} \label{s:dnr}
We report in this section detailed numerical results for each of the 50 graphs we used in our tests.
\begin{table}[tbph]
	{\footnotesize
		\caption{Clique sizes for the FDFW}
		\resizebox{\textwidth}{!}{\begin{tabular}{|l|SSS|SSS|SSS|SSS|}
				\hline
				Graph          &     & {$s=1$}  &       &     & {$s=2$}  &       &     & {$s =3$}  &       &     & {$s =4$}  &       \\ \cline{2-13} 
				& \multicolumn{1}{c}{Max} & \multicolumn{1}{c}{Mean} & \multicolumn{1}{c|}{Std} & \multicolumn{1}{c}{Max} & \multicolumn{1}{c}{Mean} & \multicolumn{1}{c|}{Std} & \multicolumn{1}{c}{Max} & \multicolumn{1}{c}{Mean} & \multicolumn{1}{c|}{Std} & \multicolumn{1}{c}{Max} & \multicolumn{1}{c}{Mean} & \multicolumn{1}{c|}{Std} \\ \hline 
				brock200\_1    & 19   & 15.6  & 1.3  & 21   & 17.0  & 1.4  & 22   & 17.9  & 1.2  & 21   & 18.6  & 1.3  \\
				brock200\_2    & 10   & 8.0   & 0.9  & 11   & 9.0   & 0.9  & 11   & 9.5   & 0.8  & 12   & 9.5   & 1.0  \\
				brock200\_3    & 13   & 10.2  & 1.1  & 14   & 11.5  & 1.0  & 14   & 12.2  & 0.9  & 14   & 12.5  & 0.9  \\
				brock200\_4    & 15   & 11.9  & 1.2  & 16   & 13.2  & 1.1  & 17   & 14.0  & 1.1  & 17   & 14.3  & 1.3  \\
				brock400\_1    & 21   & 17.5  & 1.4  & 22   & 18.8  & 1.3  & 23   & 20.0  & 1.3  & 24   & 20.9  & 1.2  \\
				brock400\_2    & 21   & 17.8  & 1.4  & 22   & 18.9  & 1.2  & 23   & 20.0  & 1.4  & 24   & 20.9  & 1.1  \\
				brock400\_3    & 21   & 17.8  & 1.5  & 22   & 19.0  & 1.3  & 23   & 20.2  & 1.2  & 24   & 20.9  & 1.2  \\
				brock400\_4    & 21   & 17.7  & 1.5  & 22   & 18.7  & 1.3  & 23   & 20.0  & 1.3  & 24   & 21.1  & 1.3  \\
				c-fat200-1     & 12   & 11.5  & 1.1  & 12   & 11.6  & 0.9  & 12   & 11.6  & 0.9  & 12   & 11.7  & 0.8  \\
				c-fat200-2     & 24   & 20.8  & 3.1  & 24   & 21.9  & 2.0  & 24   & 22.2  & 1.2  & 24   & 22.3  & 0.9  \\
				c-fat200-5     & 58   & 45.7  & 11.7 & 58   & 54.4  & 5.0  & 58   & 56.2  & 3.0  & 58   & 56.7  & 2.3  \\
				c-fat500-1     & 14   & 13.6  & 1.1  & 14   & 13.7  & 0.7  & 14   & 13.7  & 0.7  & 14   & 13.7  & 0.7  \\
				c-fat500-2     & 26   & 25.1  & 2.0  & 26   & 25.4  & 1.4  & 26   & 25.6  & 1.1  & 26   & 25.6  & 0.9  \\
				c-fat500-5     & 64   & 55.2  & 10.1 & 64   & 59.7  & 5.2  & 64   & 62.2  & 2.6  & 64   & 62.4  & 3.1  \\
				c-fat500-10    & 126  & 99.8  & 20.8 & 126  & 114.8 & 11.8 & 126  & 123.7 & 4.1  & 126  & 125.2 & 1.0  \\
				hamming6-2     & 32   & 22.6  & 4.6  & 32   & 23.3  & 4.8  & 32   & 16.2  & 5.5  & 32   & 13.7  & 5.6  \\
				hamming6-4     & 4    & 3.7   & 0.5  & 5    & 3.7   & 0.5  & 4    & 3.2   & 0.7  & 4    & 2.7   & 0.9  \\
				hamming8-2     & 121  & 83.3  & 14.5 & 122  & 86.1  & 15.1 & 123  & 87.1  & 15.3 & 123  & 88.1  & 15.2 \\
				hamming8-4     & 14   & 10.5  & 1.2  & 15   & 11.9  & 1.1  & 15   & 12.4  & 1.0  & 15   & 12.7  & 1.0  \\
				hamming10-2 \tablefootnote{The time limit was reached in 29 runs}      & 454  & 309.6 & 42.1 & 456  & 313.9 & 43.2 & 468  & 315.5 & 44.4 & 468  & 317.5 & 45.3 \\
				hamming10-4    & 31   & 27.8  & 1.3  & 32   & 29.1  & 1.3  & 33   & 30.3  & 1.2  & 34   & 31.4  & 1.3  \\
				johnson8-2-4   & 4    & 2.4   & 1.1  & 4    & 2.2   & 1.2  & 4    & 2.1   & 1.2  & 4    & 2.0   & 1.2  \\
				johnson8-4-4   & 14   & 9.5   & 1.6  & 14   & 9.9   & 1.6  & 14   & 9.9   & 1.8  & 14   & 7.6   & 2.3  \\
				johnson16-2-4  & 8    & 7.8   & 0.4  & 9    & 8.5   & 0.6  & 9    & 8.4   & 0.6  & 9    & 8.3   & 0.6  \\
				johnson32-2-4  & 16   & 14.9  & 0.7  & 17   & 15.9  & 0.8  & 17   & 16.4  & 0.7  & 18   & 17.3  & 0.8  \\
				keller4        & 11   & 8.3   & 0.9  & 12   & 9.5   & 0.7  & 12   & 10.0  & 1.0  & 13   & 9.9   & 0.9  \\
				keller5        & 20   & 17.2  & 1.1  & 20   & 18.3  & 1.0  & 21   & 19.4  & 0.9  & 23   & 20.4  & 1.0  \\
				MANN\_a9       & 16   & 14.6  & 1.0  & 17   & 11.1  & 2.9  & 17   & 9.4   & 3.1  & 17   & 8.0   & 2.9  \\
				MANN\_a27      & 118  & 117.6 & 0.5  & 119  & 118.6 & 0.6  & 120  & 119.5 & 0.7  & 121  & 120.4 & 0.7  \\
				MANN\_a45 \tablefootnote{\label{nrun2} The time limit was reached in 13 runs}     & 331  & 330.5 & 0.5  & 332  & 331.5 & 0.7  & 333  & 332.4 & 0.8  & 334  & 333.3 & 0.6  \\
				p\_hat300-1    & 8    & 6.4   & 0.8  & 9    & 7.3   & 0.7  & 9    & 7.7   & 0.7  & 9    & 7.9   & 0.8  \\
				p\_hat300-2    & 25   & 20.8  & 1.6  & 24   & 21.6  & 1.4  & 26   & 22.3  & 1.6  & 27   & 22.6  & 1.8  \\
				p\_hat300-3    & 33   & 28.9  & 1.7  & 34   & 30.1  & 1.9  & 36   & 30.9  & 1.9  & 36   & 32.0  & 1.9  \\
				p\_hat500-1    & 9    & 7.2   & 0.9  & 10   & 8.1   & 0.8  & 10   & 8.7   & 0.8  & 11   & 9.2   & 0.8  \\
				p\_hat500-2    & 35   & 29.7  & 2.1  & 35   & 30.6  & 2.1  & 37   & 31.5  & 2.2  & 37   & 32.1  & 2.2  \\
				p\_hat500-3    & 46   & 41.9  & 2.1  & 48   & 42.7  & 2.5  & 49   & 43.6  & 2.4  & 51   & 44.8  & 2.5  \\
				p\_hat700-1    & 9    & 7.4   & 0.8  & 11   & 8.4   & 0.8  & 12   & 8.9   & 0.8  & 11   & 9.6   & 0.7  \\
				p\_hat700-2    & 41   & 36.5  & 2.1  & 43   & 37.2  & 2.3  & 44   & 38.2  & 2.4  & 44   & 39.1  & 2.1  \\
				p\_hat700-3    & 57   & 52.1  & 2.5  & 60   & 53.0  & 2.5  & 61   & 54.0  & 2.3  & 61   & 55.3  & 2.7  \\
				san200\_0.7\_1 & 17   & 15.6  & 0.5  & 18   & 16.8  & 0.4  & 18   & 17.2  & 0.7  & 19   & 17.3  & 0.8  \\
				san200\_0.7\_2 & 13   & 12.9  & 0.4  & 14   & 14.0  & 0.2  & 15   & 14.5  & 0.6  & 16   & 15.0  & 0.6  \\
				san200\_0.9\_1 & 46   & 45.4  & 0.5  & 47   & 46.6  & 0.6  & 48   & 47.6  & 0.7  & 49   & 48.5  & 0.7  \\
				san200\_0.9\_2 & 39   & 35.3  & 2.2  & 43   & 36.4  & 2.4  & 41   & 37.2  & 2.4  & 42   & 38.0  & 2.6  \\
				san200\_0.9\_3 & 32   & 28.2  & 2.0  & 33   & 28.8  & 2.2  & 35   & 29.7  & 2.6  & 35   & 30.0  & 2.3  \\
				san400\_0.5\_1 & 8    & 6.5   & 0.9  & 9    & 9.0   & 0.0  & 10   & 9.9   & 0.3  & 11   & 10.5  & 0.6  \\
				san400\_0.7\_1 & 22   & 20.6  & 0.7  & 22   & 21.9  & 0.3  & 23   & 23.0  & 0.2  & 24   & 23.9  & 0.2  \\
				san400\_0.7\_2 & 16   & 15.3  & 0.7  & 17   & 17.0  & 0.2  & 18   & 18.0  & 0.0  & 19   & 18.8  & 0.4  \\
				san400\_0.7\_3 & 13   & 12.1  & 1.0  & 14   & 13.9  & 0.4  & 15   & 14.9  & 0.2  & 16   & 15.6  & 0.5  \\
				sanr200\_0.7   & 16   & 13.1  & 1.2  & 17   & 14.4  & 1.1  & 18   & 15.5  & 1.1  & 18   & 15.9  & 1.2  \\
				sanr200\_0.9   & 37   & 32.4  & 2.0  & 39   & 33.7  & 2.1  & 40   & 34.7  & 2.2  & 40   & 35.7  & 2.0 \\ \hline
		\end{tabular}}
	}
\end{table}

\begin{table}[h]
	
	\centering
	\caption{Running times for the FDFW}
	\resizebox{0.8\textwidth}{!}{\begin{tabular}{|l|SS|SS|SS|SS|}
			\hline
			Graph          &     \multicolumn{2}{c|}{$s =1$}  &        \multicolumn{2}{c|}{$s =2$}       &     \multicolumn{2}{c|}{$s =3$}   &     \multicolumn{2}{c|}{$s =4$}         \\ \cline{2-9} 
			& \multicolumn{1}{c}{Time} & \multicolumn{1}{c|}{Std} & \multicolumn{1}{c}{Time} & \multicolumn{1}{c|}{Std} & \multicolumn{1}{c}{Time} & \multicolumn{1}{c|}{Std} & \multicolumn{1}{c}{Time} & \multicolumn{1}{c|}{Std} \\ \hline 
			brock200\_1    & 0.053  & 0.0013 & 0.179  & 0.0271  & 0.227  & 0.0468 & 0.261  & 0.0341  \\
			brock200\_2    & 0.167  & 0.0297 & 0.218  & 0.0336  & 0.263  & 0.0727 & 0.327  & 0.1844  \\
			brock200\_3    & 0.156  & 0.0216 & 0.196  & 0.0350  & 0.254  & 0.0697 & 0.318  & 0.1616  \\
			brock200\_4    & 0.151  & 0.0178 & 0.189  & 0.0321  & 0.231  & 0.0483 & 0.278  & 0.0777  \\
			brock400\_1    & 0.677  & 0.0989 & 0.968  & 0.0246  & 1.250  & 0.1431 & 1.495  & 0.0679  \\
			brock400\_2    & 0.591  & 0.0197 & 0.983  & 0.0939  & 1.258  & 0.1185 & 1.493  & 0.0812  \\
			brock400\_3    & 0.595  & 0.0201 & 0.972  & 0.0316  & 1.242  & 0.0518 & 1.493  & 0.1350  \\
			brock400\_4    & 0.597  & 0.0190 & 0.961  & 0.0352  & 1.258  & 0.1030 & 1.519  & 0.1454  \\
			c-fat200-1     & 0.181  & 0.0057 & 0.318  & 0.0080  & 0.348  & 0.0082 & 0.376  & 0.0119  \\
			c-fat200-2     & 0.180  & 0.0079 & 0.314  & 0.0123  & 0.345  & 0.0169 & 0.379  & 0.0208  \\
			c-fat200-5     & 0.148  & 0.0071 & 0.264  & 0.0174  & 0.317  & 0.0247 & 0.356  & 0.0251  \\
			c-fat500-1     & 1.488  & 0.0249 & 2.441  & 0.0446  & 2.728  & 0.0458 & 3.016  & 0.0515  \\
			c-fat500-2     & 1.519  & 0.0257 & 2.579  & 0.1139  & 2.891  & 0.1459 & 3.256  & 0.1961  \\
			c-fat500-5     & 1.600  & 0.0641 & 2.408  & 0.1294  & 2.868  & 0.2600 & 3.311  & 0.4613  \\
			c-fat500-10    & 1.664  & 0.1425 & 2.436  & 0.3427  & 2.719  & 0.5139 & 3.073  & 0.5811  \\
			hamming6-2     & 0.008  & 0.0011 & 0.015  & 0.0036  & 0.020  & 0.0037 & 0.026  & 0.0107  \\
			hamming6-4     & 0.013  & 0.0005 & 0.024  & 0.0009  & 0.025  & 0.0018 & 0.027  & 0.0013  \\
			hamming8-2     & 0.172  & 0.0350 & 0.267  & 0.0515  & 0.346  & 0.0604 & 0.428  & 0.0645  \\
			hamming8-4     & 0.172  & 0.0296 & 0.312  & 0.0283  & 0.368  & 0.0487 & 0.409  & 0.0330  \\
			hamming10-2  & 19.686 & 8.3243 & 21.857 & 10.8597 & 22.468 & 9.8968 & 25.991 & 11.2051 \\
			hamming10-4    & 5.295  & 0.4853 & 8.365  & 0.5668  & 10.680 & 0.7274 & 13.257 & 0.7791  \\
			johnson8-2-4   & 0.003  & 0.0005 & 0.005  & 0.0005  & 0.007  & 0.0030 & 0.009  & 0.0049  \\
			johnson8-4-4   & 0.011  & 0.0009 & 0.020  & 0.0044  & 0.024  & 0.0033 & 0.030  & 0.0067  \\
			johnson16-2-4  & 0.037  & 0.0048 & 0.065  & 0.0048  & 0.080  & 0.0030 & 0.097  & 0.0041  \\
			johnson32-2-4  & 1.235  & 0.3833 & 1.547  & 0.2246  & 1.900  & 0.1825 & 2.325  & 0.2487  \\
			keller4        & 0.081  & 0.0163 & 0.164  & 0.0733  & 0.209  & 0.1159 & 0.253  & 0.1667  \\
			keller5        & 2.354  & 0.3020 & 3.819  & 0.6645  & 5.140  & 1.3110 & 6.438  & 2.0281  \\
			MANN\_a9       & 0.005  & 0.0013 & 0.012  & 0.0081  & 0.018  & 0.0132 & 0.027  & 0.0260  \\
			MANN\_a27      & 1.757  & 0.3220 & 2.504  & 0.5624  & 3.362  & 0.7169 & 4.371  & 0.8807  \\
			MANN\_a45  & 44.945 & 5.8122 & 54.466 & 8.6046  & 67.499 & 9.7882 & 85.228 & 11.6795 \\
			p\_hat300-1    & 0.348  & 0.0133 & 0.600  & 0.0408  & 0.634  & 0.0592 & 0.683  & 0.0746  \\
			p\_hat300-2    & 0.323  & 0.0105 & 0.548  & 0.0177  & 0.631  & 0.0246 & 0.717  & 0.0754  \\
			p\_hat300-3    & 0.262  & 0.0182 & 0.401  & 0.0197  & 0.466  & 0.0237 & 0.534  & 0.0332  \\
			p\_hat500-1    & 1.209  & 0.0944 & 1.996  & 0.2460  & 2.293  & 0.8532 & 2.449  & 0.5990  \\
			p\_hat500-2    & 1.353  & 0.0713 & 2.295  & 0.0999  & 2.836  & 0.1464 & 3.261  & 0.1831  \\
			p\_hat500-3    & 1.369  & 0.1112 & 2.161  & 0.1798  & 2.535  & 0.1734 & 2.803  & 0.1708  \\
			p\_hat700-1    & 2.549  & 0.1642 & 4.080  & 0.3388  & 4.468  & 0.8064 & 4.943  & 1.0512  \\
			p\_hat700-2    & 2.765  & 0.1417 & 5.040  & 0.2496  & 6.340  & 0.2623 & 7.396  & 0.3281  \\
			p\_hat700-3    & 3.084  & 0.2889 & 4.929  & 0.5101  & 5.871  & 0.5302 & 6.607  & 0.5871  \\
			san200\_0.7\_1 & 0.159  & 0.0454 & 0.277  & 0.0429  & 0.358  & 0.0659 & 0.481  & 0.0947  \\
			san200\_0.7\_2 & 0.271  & 0.0732 & 0.402  & 0.1320  & 0.404  & 0.1214 & 0.452  & 0.0791  \\
			san200\_0.9\_1 & 0.128  & 0.0218 & 0.220  & 0.0302  & 0.296  & 0.0282 & 0.373  & 0.0332  \\
			san200\_0.9\_2 & 0.129  & 0.0171 & 0.207  & 0.0225  & 0.260  & 0.0278 & 0.319  & 0.0295  \\
			san200\_0.9\_3 & 0.105  & 0.0134 & 0.183  & 0.0186  & 0.237  & 0.0392 & 0.319  & 0.0786  \\
			san400\_0.5\_1 & 1.605  & 0.1003 & 2.178  & 0.5510  & 2.011  & 0.6413 & 2.860  & 0.8604  \\
			san400\_0.7\_1 & 1.706  & 0.6290 & 2.266  & 0.6993  & 2.469  & 0.1244 & 2.956  & 0.1424  \\
			san400\_0.7\_2 & 1.684  & 0.5364 & 2.069  & 0.7020  & 2.133  & 0.3488 & 2.497  & 0.1083  \\
			san400\_0.7\_3 & 1.713  & 0.3159 & 2.131  & 0.7214  & 2.020  & 0.6281 & 2.244  & 0.4235  \\
			sanr200\_0.7   & 0.103  & 0.0045 & 0.186  & 0.0248  & 0.227  & 0.0476 & 0.297  & 0.0961  \\
			sanr200\_0.9   & 0.108  & 0.0110 & 0.173  & 0.0175  & 0.220  & 0.0178 & 0.283  & 0.0223 \\ \hline
	\end{tabular}}
	
	\label{tab:2}
\end{table}

\clearpage

\begin{table}[h]
	\caption{Clique sizes for the FWdc}
	\resizebox{\textwidth}{!}{\begin{tabular}{|l|SSS|SSS|SSS|SSS|}
			\hline
			Graph          &     & {$s =1$}  &       &     & {$s =2$}  &       &     & {$s =3$}  &       &     & {$s =4$}  &       \\ \cline{2-13} 
			& \multicolumn{1}{c}{Max} & \multicolumn{1}{c}{Mean} & \multicolumn{1}{c|}{Std} & \multicolumn{1}{c}{Max} & \multicolumn{1}{c}{Mean} & \multicolumn{1}{c|}{Std} & \multicolumn{1}{c}{Max} & \multicolumn{1}{c}{Mean} & \multicolumn{1}{c|}{Std} & \multicolumn{1}{c}{Max} & \multicolumn{1}{c}{Mean} & \multicolumn{1}{c|}{Std} \\ \hline 
			brock200\_1    & 21   & 18.2  & 1.01  & 21   & 18.2  & 0.90  & 21   & 18.5  & 1.02  & 22  & 18.7  & 1.05  \\
			brock200\_2    & 10   & 8.6   & 0.91  & 11   & 8.9   & 0.81  & 11   & 9.3   & 0.91  & 12  & 9.5   & 0.86  \\
			brock200\_3    & 13   & 11.4  & 0.87  & 14   & 11.5  & 0.88  & 15   & 11.9  & 1.02  & 14  & 12.0  & 0.97  \\
			brock200\_4    & 16   & 13.3  & 1.01  & 16   & 13.7  & 0.93  & 16   & 14.0  & 1.01  & 17  & 14.1  & 1.14  \\
			brock400\_1    & 24   & 20.7  & 1.10  & 25   & 21.4  & 1.30  & 24   & 21.3  & 1.25  & 25  & 21.7  & 1.33  \\
			brock400\_2    & 24   & 20.9  & 1.12  & 25   & 21.3  & 1.26  & 26   & 21.4  & 1.32  & 25  & 21.8  & 1.16  \\
			brock400\_3    & 24   & 20.7  & 0.97  & 24   & 21.0  & 1.18  & 25   & 21.2  & 1.09  & 24  & 21.4  & 1.06  \\
			brock400\_4    & 23   & 20.6  & 1.16  & 23   & 21.1  & 1.07  & 24   & 21.5  & 1.07  & 25  & 21.6  & 1.38  \\
			c-fat200-1     & 12   & 11.4  & 1.69  & 12   & 10.8  & 2.64  & 12   & 8.7   & 4.31  & 12  & 7.9   & 4.49  \\
			c-fat200-2     & 24   & 21.2  & 4.76  & 24   & 20.3  & 6.25  & 24   & 18.7  & 7.63  & 24  & 17.2  & 8.45  \\
			c-fat200-5     & 58   & 55.1  & 7.08  & 58   & 53.3  & 9.88  & 58   & 52.0  & 12.88 & 58  & 53.3  & 10.66 \\
			c-fat500-1     & 14   & 13.5  & 1.13  & 14   & 12.7  & 2.69  & 14   & 10.8  & 4.57  & 14  & 9.5   & 5.24  \\
			c-fat500-2     & 26   & 25.5  & 2.47  & 26   & 24.5  & 5.23  & 26   & 22.8  & 7.65  & 26  & 21.6  & 8.73  \\
			c-fat500-5     & 64   & 60.8  & 10.85 & 64   & 61.7  & 8.62  & 64   & 58.8  & 14.73 & 64  & 56.2  & 18.38 \\
			c-fat500-10    & 126  & 122.6 & 12.02 & 126  & 119.8 & 17.20 & 126  & 118.0 & 24.57 & 126 & 115.7 & 29.14 \\
			hamming6-2     & 32   & 28.6  & 4.58  & 32   & 27.9  & 4.63  & 32   & 27.5  & 4.40  & 32  & 27.4  & 4.09  \\
			hamming6-4     & 4    & 3.7   & 0.46  & 5    & 4.2   & 0.61  & 6    & 4.4   & 0.66  & 6   & 4.8   & 0.63  \\
			hamming8-2     & 128  & 121.1 & 9.18  & 128  & 120.2 & 9.07  & 128  & 118.8 & 10.29 & 128 & 116.9 & 11.86 \\
			hamming8-4     & 16   & 12.7  & 2.68  & 16   & 12.6  & 2.31  & 16   & 12.5  & 2.24  & 17  & 12.8  & 2.28  \\
			hamming10-2    & 512  & 498.9 & 14.52 & 512  & 497.0 & 15.59 & 512  & 495.5 & 17.16 & 512 & 493.9 & 18.06 \\
			hamming10-4    & 36   & 31.6  & 2.93  & 36   & 32.2  & 2.67  & 37   & 32.1  & 2.80  & 37  & 32.9  & 2.32  \\
			johnson8-2-4   & 4    & 4.0   & 0.00  & 5    & 4.9   & 0.27  & 5    & 5.0   & 0.10  & 6   & 5.3   & 0.47  \\
			johnson8-4-4   & 14   & 11.9  & 1.73  & 14   & 11.7  & 1.46  & 14   & 11.8  & 1.35  & 15  & 11.9  & 1.18  \\
			johnson16-2-4  & 8    & 8.0   & 0.00  & 9    & 9.0   & 0.00  & 9    & 9.0   & 0.00  & 10  & 9.8   & 0.39  \\
			johnson32-2-4  & 16   & 16.0  & 0.00  & 17   & 17.0  & 0.00  & 17   & 17.0  & 0.00  & 18  & 17.8  & 0.41  \\
			keller4        & 12   & 9.3   & 1.15  & 12   & 9.7   & 0.81  & 13   & 10.1  & 0.80  & 13  & 10.6  & 0.82  \\
			keller5        & 27   & 20.7  & 1.77  & 26   & 21.1  & 1.72  & 26   & 21.5  & 1.50  & 27  & 21.5  & 1.55  \\
			MANN\_a9       & 17   & 16.3  & 0.67  & 18   & 16.8  & 0.71  & 19   & 17.4  & 0.72  & 19  & 17.6  & 0.82  \\
			MANN\_a27      & 120  & 118.2 & 0.43  & 120  & 119.2 & 0.37  & 121  & 120.1 & 0.43  & 122 & 121.1 & 0.38  \\
			MANN\_a45      & 332  & 331.0 & 0.17  & 333  & 332.0 & 0.17  & 334  & 333.0 & 0.20  & 335 & 334.0 & 0.22  \\
			p\_hat300-1    & 8    & 6.9   & 0.68  & 9    & 7.1   & 0.79  & 9    & 7.4   & 0.75  & 9   & 7.5   & 0.85  \\
			p\_hat300-2    & 26   & 21.9  & 1.25  & 25   & 22.0  & 1.20  & 25   & 22.2  & 1.18  & 26  & 22.2  & 1.27  \\
			p\_hat300-3    & 35   & 31.4  & 1.35  & 34   & 31.8  & 1.32  & 35   & 31.8  & 1.30  & 36  & 32.2  & 1.19  \\
			p\_hat500-1    & 10   & 7.9   & 0.81  & 10   & 8.1   & 0.83  & 10   & 8.1   & 0.90  & 11  & 8.4   & 0.96  \\
			p\_hat500-2    & 35   & 31.6  & 1.84  & 35   & 31.9  & 1.73  & 36   & 31.8  & 1.79  & 35  & 31.8  & 1.57  \\
			p\_hat500-3    & 48   & 44.8  & 1.58  & 49   & 44.8  & 1.76  & 49   & 45.1  & 1.73  & 49  & 45.4  & 1.67  \\
			p\_hat700-1    & 9    & 8.0   & 0.70  & 10   & 8.2   & 0.78  & 10   & 8.4   & 0.78  & 10  & 8.6   & 0.77  \\
			p\_hat700-2    & 44   & 39.9  & 1.77  & 43   & 40.0  & 1.77  & 44   & 40.2  & 1.81  & 44  & 40.0  & 2.00  \\
			p\_hat700-3    & 62   & 57.0  & 1.94  & 60   & 57.6  & 1.73  & 61   & 57.5  & 1.85  & 62  & 57.9  & 1.94  \\
			san200\_0.7\_1 & 18   & 16.6  & 0.84  & 19   & 17.1  & 1.23  & 20   & 18.0  & 1.17  & 21  & 18.7  & 1.44  \\
			san200\_0.7\_2 & 15   & 13.1  & 0.26  & 15   & 14.0  & 0.14  & 16   & 14.7  & 0.49  & 16  & 15.2  & 0.42  \\
			san200\_0.9\_1 & 65   & 48.9  & 4.23  & 65   & 49.4  & 3.94  & 68   & 50.1  & 4.11  & 70  & 50.6  & 3.85  \\
			san200\_0.9\_2 & 52   & 39.5  & 2.74  & 52   & 39.9  & 2.53  & 55   & 40.8  & 3.32  & 55  & 41.5  & 3.35  \\
			san200\_0.9\_3 & 36   & 33.4  & 1.24  & 36   & 33.7  & 1.30  & 38   & 34.1  & 1.43  & 39  & 34.4  & 1.51  \\
			san400\_0.5\_1 & 9    & 8.1   & 0.30  & 10   & 9.0   & 0.32  & 10   & 9.6   & 0.50  & 12  & 10.1  & 0.57  \\
			san400\_0.7\_1 & 23   & 21.8  & 0.77  & 24   & 22.5  & 0.86  & 25   & 22.9  & 1.30  & 25  & 23.5  & 1.73  \\
			san400\_0.7\_2 & 23   & 17.4  & 1.16  & 20   & 17.9  & 0.78  & 21   & 18.5  & 0.93  & 21  & 19.0  & 0.75  \\
			san400\_0.7\_3 & 17   & 15.1  & 0.93  & 18   & 15.6  & 0.85  & 18   & 16.2  & 0.92  & 19  & 16.6  & 0.88  \\
			sanr200\_0.7   & 17   & 14.9  & 0.86  & 18   & 15.2  & 1.09  & 17   & 15.6  & 0.89  & 19  & 15.8  & 0.99  \\
			sanr200\_0.9   & 41   & 37.5  & 1.79  & 41   & 37.5  & 1.67  & 42   & 38.1  & 1.75  & 43  & 38.3  & 1.73 \\ \hline
	\end{tabular}}
	
\end{table}	
\pagebreak
\begin{table}[h]
	\caption{Running times for the FWdc}	
	\resizebox{\textwidth}{!}{\begin{tabular}{|l|SS|SS|SS|SS|}
			\hline
			Graph          &     \multicolumn{2}{c|}{$s =1$}  &        \multicolumn{2}{c|}{$s =2$}       &     \multicolumn{2}{c|}{$s =3$}   &     \multicolumn{2}{c|}{$s =4$}         \\ \cline{2-9} 
			& \multicolumn{1}{c}{Time} & \multicolumn{1}{c|}{Std} & \multicolumn{1}{c}{Time} & \multicolumn{1}{c|}{Std} & \multicolumn{1}{c}{Time} & \multicolumn{1}{c|}{Std} & \multicolumn{1}{c}{Time} & \multicolumn{1}{c|}{Std} \\ \hline 
			brock200\_1    & 0.0060 & 0.00560 & 0.0060 & 0.00064 & 0.0064 & 0.00070 & 0.0069 & 0.00083 \\
			brock200\_2    & 0.0045 & 0.00053 & 0.0051 & 0.00041 & 0.0051 & 0.00050 & 0.0053 & 0.00042 \\
			brock200\_3    & 0.0044 & 0.00045 & 0.0050 & 0.00056 & 0.0052 & 0.00047 & 0.0052 & 0.00038 \\
			brock200\_4    & 0.0045 & 0.00039 & 0.0053 & 0.00055 & 0.0055 & 0.00054 & 0.0057 & 0.00052 \\
			brock400\_1    & 0.0122 & 0.00088 & 0.0142 & 0.00157 & 0.0141 & 0.00082 & 0.0144 & 0.00080 \\
			brock400\_2    & 0.0129 & 0.00066 & 0.0133 & 0.00118 & 0.0144 & 0.00095 & 0.0147 & 0.00072 \\
			brock400\_3    & 0.0129 & 0.00060 & 0.0138 & 0.00098 & 0.0139 & 0.00057 & 0.0146 & 0.00096 \\
			brock400\_4    & 0.0130 & 0.00065 & 0.0136 & 0.00100 & 0.0142 & 0.00069 & 0.0146 & 0.00067 \\
			c-fat200-1     & 0.0065 & 0.00061 & 0.0070 & 0.00080 & 0.0068 & 0.00052 & 0.0071 & 0.00074 \\
			c-fat200-2     & 0.0060 & 0.00100 & 0.0067 & 0.00109 & 0.0067 & 0.00098 & 0.0067 & 0.00091 \\
			c-fat200-5     & 0.0071 & 0.00216 & 0.0081 & 0.00237 & 0.0084 & 0.00261 & 0.0089 & 0.00259 \\
			c-fat500-1     & 0.0351 & 0.00254 & 0.0370 & 0.00188 & 0.0366 & 0.00210 & 0.0364 & 0.00243 \\
			c-fat500-2     & 0.0311 & 0.00360 & 0.0331 & 0.00374 & 0.0334 & 0.00355 & 0.0334 & 0.00314 \\
			c-fat500-5     & 0.0296 & 0.00621 & 0.0330 & 0.00730 & 0.0343 & 0.00599 & 0.0347 & 0.00592 \\
			c-fat500-10    & 0.0390 & 0.01353 & 0.0400 & 0.01617 & 0.0422 & 0.01685 & 0.0433 & 0.01748 \\
			hamming6-2     & 0.0014 & 0.00012 & 0.0021 & 0.00017 & 0.0023 & 0.00027 & 0.0026 & 0.00029 \\
			hamming6-4     & 0.0009 & 0.00015 & 0.0011 & 0.00014 & 0.0012 & 0.00014 & 0.0013 & 0.00015 \\
			hamming8-2     & 0.0204 & 0.00162 & 0.0245 & 0.00195 & 0.0287 & 0.00272 & 0.0302 & 0.00308 \\
			hamming8-4     & 0.0060 & 0.00055 & 0.0065 & 0.00059 & 0.0068 & 0.00075 & 0.0069 & 0.00075 \\
			hamming10-2    & 1.0562 & 0.06427 & 1.0825 & 0.06712 & 1.0816 & 0.07577 & 1.0717 & 0.07835 \\
			hamming10-4    & 0.0762 & 0.00268 & 0.0805 & 0.00400 & 0.0805 & 0.00379 & 0.0823 & 0.00342 \\
			johnson8-2-4   & 0.0004 & 0.00006 & 0.0006 & 0.00009 & 0.0007 & 0.00013 & 0.0008 & 0.00013 \\
			johnson8-4-4   & 0.0011 & 0.00009 & 0.0015 & 0.00017 & 0.0018 & 0.00021 & 0.0020 & 0.00027 \\
			johnson16-2-4  & 0.0023 & 0.00028 & 0.0029 & 0.00036 & 0.0034 & 0.00041 & 0.0037 & 0.00055 \\
			johnson32-2-4  & 0.0218 & 0.00218 & 0.0240 & 0.00229 & 0.0248 & 0.00232 & 0.0259 & 0.00282 \\
			keller4        & 0.0034 & 0.00034 & 0.0044 & 0.00053 & 0.0044 & 0.00065 & 0.0051 & 0.00062 \\
			keller5        & 0.0428 & 0.00207 & 0.0446 & 0.00206 & 0.0455 & 0.00239 & 0.0460 & 0.00236 \\
			MANN\_a9       & 0.0010 & 0.00013 & 0.0014 & 0.00017 & 0.0016 & 0.00024 & 0.0018 & 0.00031 \\
			MANN\_a27      & 0.0910 & 0.00610 & 0.1029 & 0.00723 & 0.1096 & 0.00723 & 0.1150 & 0.00699 \\
			MANN\_a45      & 2.3774 & 0.07529 & 2.4419 & 0.10145 & 2.4544 & 0.08325 & 2.4712 & 0.09195 \\
			p\_hat300-1    & 0.0071 & 0.00029 & 0.0078 & 0.00068 & 0.0080 & 0.00046 & 0.0083 & 0.00040 \\
			p\_hat300-2    & 0.0085 & 0.00031 & 0.0096 & 0.00061 & 0.0103 & 0.00080 & 0.0104 & 0.00062 \\
			p\_hat300-3    & 0.0095 & 0.00046 & 0.0110 & 0.00070 & 0.0119 & 0.00072 & 0.0126 & 0.00092 \\
			p\_hat500-1    & 0.0163 & 0.00041 & 0.0169 & 0.00039 & 0.0170 & 0.00043 & 0.0171 & 0.00057 \\
			p\_hat500-2    & 0.0208 & 0.00129 & 0.0224 & 0.00133 & 0.0228 & 0.00133 & 0.0235 & 0.00131 \\
			p\_hat500-3    & 0.0243 & 0.00085 & 0.0266 & 0.00100 & 0.0278 & 0.00125 & 0.0281 & 0.00100 \\
			p\_hat700-1    & 0.0298 & 0.00070 & 0.0307 & 0.00091 & 0.0309 & 0.00109 & 0.0311 & 0.00079 \\
			p\_hat700-2    & 0.0394 & 0.00108 & 0.0421 & 0.00211 & 0.0427 & 0.00143 & 0.0431 & 0.00161 \\
			p\_hat700-3    & 0.0466 & 0.00151 & 0.0498 & 0.00181 & 0.0508 & 0.00177 & 0.0520 & 0.00144 \\
			san200\_0.7\_1 & 0.0049 & 0.00059 & 0.0056 & 0.00071 & 0.0061 & 0.00066 & 0.0064 & 0.00068 \\
			san200\_0.7\_2 & 0.0048 & 0.00068 & 0.0063 & 0.00093 & 0.0073 & 0.00134 & 0.0081 & 0.00148 \\
			san200\_0.9\_1 & 0.0086 & 0.00076 & 0.0108 & 0.00104 & 0.0118 & 0.00107 & 0.0129 & 0.00149 \\
			san200\_0.9\_2 & 0.0075 & 0.00053 & 0.0092 & 0.00075 & 0.0101 & 0.00079 & 0.0113 & 0.00073 \\
			san200\_0.9\_3 & 0.0070 & 0.00058 & 0.0084 & 0.00067 & 0.0092 & 0.00069 & 0.0098 & 0.00084 \\
			san400\_0.5\_1 & 0.0116 & 0.00056 & 0.0128 & 0.00094 & 0.0141 & 0.00172 & 0.0142 & 0.00136 \\
			san400\_0.7\_1 & 0.0130 & 0.00056 & 0.0143 & 0.00077 & 0.0152 & 0.00101 & 0.0158 & 0.00110 \\
			san400\_0.7\_2 & 0.0128 & 0.00064 & 0.0139 & 0.00068 & 0.0147 & 0.00082 & 0.0157 & 0.00138 \\
			san400\_0.7\_3 & 0.0125 & 0.00043 & 0.0136 & 0.00059 & 0.0143 & 0.00072 & 0.0148 & 0.00084 \\
			sanr200\_0.7   & 0.0049 & 0.00061 & 0.0057 & 0.00071 & 0.0058 & 0.00064 & 0.0062 & 0.00078 \\
			sanr200\_0.9   & 0.0070 & 0.00053 & 0.0087 & 0.00072 & 0.0095 & 0.00078 & 0.0119 & 0.00090 \\ \hline
	\end{tabular}}
\end{table}	

\clearpage

\bibliographystyle{plain}	
\bibliography{cubic}
\end{document}